\documentclass[11pt]{article}
\usepackage{bbm}

\usepackage{amsfonts}

\usepackage{mathrsfs,amssymb,amsmath,amsfonts}
\usepackage{amsthm}
\usepackage{enumerate}
\usepackage{color}
\usepackage[dvipdfm]{hyperref}% Latex-Dvi-Pdf hyperlink

\newtheorem{theorem}{Theorem}[section]
\newtheorem{lemma}[theorem]{Lemma}

\newtheorem{proposition}[theorem]{Proposition}
\newtheorem{corollary}[theorem]{Corollary}
\newtheorem{remark}[theorem]{Remark}

\theoremstyle{definition}
\newtheorem{example}[theorem]{Example}

\newenvironment{keywords}{{\bf Key words: }}{}
\newenvironment{AMS}{{\bf AMS subject classification: }}{}

\numberwithin{equation}{section}

\newcommand{\norm}[1]{\left\Vert#1\right\Vert}

\newcommand{\tr}{\mathop{\mathrm{tr}}}
\newcommand{\diag}{\mathop{\mathrm{diag}}}

\begin{document}

\title{Infinite horizon jump-diffusion forward-backward stochastic
differential equations and their application to backward
linear-quadratic problems \footnote{This work is supported by the
National Natural Science Foundation of China (11471192) and the
Natural Science Foundation of Shandong Province (JQ201401).} }

\author{Zhiyong Yu \footnote{Email: {\tt
yuzhiyong@sdu.edu.cn}}\\
{\small School of Mathematics, Shandong University, Jinan 250100,
China}}

\maketitle

\begin{abstract}
In this paper, we investigate infinite horizon jump-diffusion
forward-backward stochastic differential equations under some
monotonicity conditions. We establish an existence and uniqueness
theorem, two stability results and a comparison theorem for
solutions to such kind of equations. Then the theoretical results
are applied to study a kind of infinite horizon backward stochastic
linear-quadratic optimal control problems, and then differential
game problems. The unique optimal controls for the control problems
and the unique Nash equilibrium points for the game problems are
obtained in closed forms.
\end{abstract}

\begin{keywords}
Forward-backward stochastic differential equation, monotonicity
condition, stochastic optimal control, nonzero-sum stochastic
differential game, linear-quadratic problem.
\end{keywords}\\

\begin{AMS}
60H10, 93E20, 49N10.
\end{AMS}

\section{Introduction}

%The linear backward stochastic differential equations (SDEs) were
%originally introduced by Bismut \cite{Bismut1978} to define the
%adjoint processes in stochastic optimal control theory, and the
%nonlinear ones were first studied by Pardoux and Peng
%\cite{PardouxPeng1990}. Since then, the theory of backward SDEs were
%widely studied and its applications have been found in many fields.
%Besides the application in the stochastic optimization theory, we
%would like to mention another one, that backward SDEs are used to
%give a probabilistic representation for semilinear second order
%partial differential equations (PDEs) of elliptic or parabolic type,
%which generalized the classical Feynman-Kac formula for linear PDEs
%(see for example Peng \cite{Peng1991}, Pardoux and Peng
%\cite{PardouxPeng1992}).

Coupled forward-backward stochastic differential equations (SDEs)
are encountered when one applies the classical stochastic maximum
principle to optimal control or differential game problems (see Ma
and Yong \cite{MaYong1999}, Yong and Zhou \cite{YongZhou1999}). The
existence and uniqueness of solutions to such kind of equations are
closely linked to that of optimal controls or Nash equilibrium
points. Forward-backward SDEs are also used to give a probabilistic
interpretation for quasilinear second order partial differential
equations (PDEs) of elliptic or parabolic type (see Pardoux and Tang
\cite{PardouxTang1999}, Wu and Yu \cite{WuYu2014}), which
generalized the classical Feynman-Kac formula for linear PDEs.
Moreover, in mathematical finance forward-backward SDEs are often
adopted to describe the models involving large investors (see for
example Cvitani\'{c} and Ma \cite{CvitanicMa1996}).

Finite horizon forward-backward SDEs were first investigated by
Antonelli \cite{Antonelli1993} and a local existence and uniqueness
result was obtained. He also constructed a counterexample showing
that, a large time duration might lead to non-solvability just under
the Lipschitz assumption. For the global solvability results, three
fundamental methods are available. The first one is the {\it method
of contraction mapping} used by Pardoux and Tang
\cite{PardouxTang1999}. The second one concerns a kind of {\it
four-step scheme} approach introduced by Ma, Protter and Yong
\cite{MaProtterYong1994} which can be regarded as a combination of
the methods of PDEs and probability theory. This method requires the
non-degeneracy of the forward diffusion, and is only effective in
Markovian frameworks. The third one called {\it method of
continuation} is probabilistic. This method gets rid of the
restriction of non-degeneracy, and can deal with non-Markovian
forward-backward SDEs. As a trade-off, a kind of monotonicity
conditions on the coefficients is introduced to ensure the
solvability, which is restrictive in a different way. This method is
initiated by Hu and Peng \cite{HuPeng1995}, Peng and Wu
\cite{PengWu1999}. Later, Yong \cite{Yong1997,Yong2010} made
improvements and made the method more systematic. For some recent
developments on finite horizon forward-backward SDEs, one can refer
to Ma et al. \cite{MWZZ2011}.

In 2000, Peng and Shi \cite{PengShi2000}, for the first time,
studied infinite horizon forward-backward SDEs driven by Brownian
motions employing the method of continuation. Later, Wu
\cite{Wu2003} studied this problem in some different monotonicity
framework from \cite{PengShi2000}. Yin \cite{Yin2008,Yin2011}
investigates the same issue by the method of contraction mapping.
Some existence and uniqueness results and comparison theorems were
obtained. In this paper, we consider a kind of generalized infinite
horizon forward-backward SDEs driven by both Brownian motions and
Poisson processes as follows:
\begin{equation}\label{Sec1_FBSDE}
\left\{
\begin{aligned}
dx(t) =\ & b(t,x(t),y(t),z(t),r(t,\cdot)) dt
+\sigma(t,x(t),y(t),z(t),r(t,\cdot)) dW(t)\\
& +\int_{\mathcal E} \gamma(t,e,x(t-),y(t-),z(t),r(t,e)) \tilde
N(dt,de),\\
-dy(t) =\ & g(t,x(t),y(t),z(t),r(t,\cdot)) dt -z(t) dW(t)
-\int_{\mathcal E} r(t,e) \tilde N(dt, de),\\
x(0) =\ & \Phi(y(0)),
\end{aligned}
\right.
\end{equation}
where the notations and mappings will be given in Section \ref{Sec2}
and Section \ref{Sec3}. We adopt the model with random jumps, since
jump-diffusion processes characterize stochastic phenomena more
often and accurate than just diffusion processes, which provide us
with more realistic models in practice. For example, in finance,
stock prices often exhibit some jump behaviors. Moreover, financial
markets with jump stock prices provide a rich family of incomplete
financial models. For more information about jump diffusion models,
the interested readers may be referred to Cont and Tankov
\cite{ContTankov2004}, {\O}ksendal and Sulem \cite{OS2007}, Shen,
Meng and Shi \cite{ShenMengShi2014}.

%A special situation of \eqref{Sec1_FBSDE}, where $\sigma =0$,
%$\gamma=0$, $z=0$, $r=0$, $b= H_y(t,x,y)$, $g= H_x(t,x,y)$ and
%$\Phi=x_0$, is the classical Hamiltonian system with infinite
%horizon:
%\begin{equation}\label{Sec1_Hamilton_Sys}
%\left\{
%\begin{aligned}
%x'(t) =\ & H_y(t,x(t),y(t)),\\
%y'(t) =\ & -H_x(t,x(t),y(t)),\\
%x(0) =\ & x_0.
%\end{aligned}
%\right.
%\end{equation}
%Moreover, inspirited by the classical variation method as well as
%the stochastic maximum principle in optimal control theory, the
%above deterministic Hamiltonian system allows a generalized
%stochastic version, namely stochastic Hamiltonian system with
%infinite horizon:
%\begin{equation}\label{Sec1_Stoc_Hamilton_Sys}
%\left\{
%\begin{aligned}
%dx(t) =\ & H_y(t,x(t),y(t),z(t),r(t,e)) dt +
%H_z(t,x(t),y(t),z(t),r(t,e)) dW(t)\\
%& + \int_{\mathcal E} H_r(t,x(t),y(t),z(t),r(t,e)) \tilde
%N(dt,de),\\
%-dy(t) =\ & H_x(t,x(t),y(t),z(t),r(t,e)) dt -z(t)dW(t)
%-\int_{\mathcal E} r(t,e) \tilde N(dt,de),\\
%x(0) =\ & \Phi(y(0))
%\end{aligned}
%\right.
%\end{equation}
%(see Yong and Zhou \cite{YongZhou1999}), which is also a special
%case of \eqref{Sec1_FBSDE}. This relation to stochastic Hamiltonian
%systems and stochastic optimization problems can be regarded as our
%main motivation.

In this paper we study the solvability of forward-backward SDEs by
virtue of the method of continuation. The idea is to introduce a
family of infinite horizon forward-backward SDEs parameterized by
$\alpha\in [0,1]$ such that, when $\alpha=1$ the forward-backward
SDE coincides with \eqref{Sec1_FBSDE} and when $\alpha=0$ the
forward-backward SDE is uniquely solvable. We will show that there
exists a fixed step-length $\delta_0>0$, such that, if, for some
$\alpha_0\in [0,1)$, the parameterized forward-backward SDE is
uniquely solvable, then the same conclusion holds for $\alpha_0$
being replaced by $\alpha_0+\delta\leq 1$ with $\delta\in
[0,\delta_0]$. Once this has been proved, we can increase the
parameter $\alpha$ step by step and finally reach $\alpha=1$.

It is worth noting that, we study a kind of more general coupled
forward-backward SDEs in comparison with
\cite{PengShi2000,Wu2003,Yin2008,Yin2011}. Besides the coupling in
$b, \sigma, \gamma$ and $g$, in this paper the initial values are
also in a coupled form: $x(0) = \Phi(y(0))$ (see
\eqref{Sec1_FBSDE}). The traditional technique treating the coupling
in the initial values (or terminal values) when the horizon is
finite (see for example \cite{HuPeng1995,Yong1997,PengWu1999}) is to
parameterize and analyze the initial coefficient $\Phi$ as the same
as other coefficients $(b,\sigma,\gamma,g)$. When this traditional
technique is used to the case of infinite horizon, we can solve two
special cases: (i) the decoupled case: $\Phi(y(0))=x_0$; (ii) the
strong monotonicity case: there exists a constant $\nu>0$ such that,
for any $y_1,y_2\in\mathbb R^n$, $\langle \Phi(y_1)-\Phi(y_2),\
y_1-y_2 \rangle \leq -\nu|y_1-y_2|^2$. However, in many practical
stochastic optimization problems, these two kinds of conditions are
too strong to satisfy naturally. In the present paper, instead of
the traditional parameterization technique, we employ the classical
mean value theorem of continuous functions and some delicate
techniques to handle the coupling between the two initial values.
This technique was introduced for the first time by Wu and Yu
\cite{WuYu2014} to analyze some algebraic equations. By virtue of
the new technique, the conclusion is improved to a general
monotonicity case: $\langle \Phi(y_1)-\Phi(y_2),\ y_1-y_2 \rangle
\leq 0$ (see (H3.3)-(ii)) which is natural in the viewpoint of
practical optimization problems. A potential application of the new
technique is to deal with the associated finite horizon problems and
hope to improve the corresponding results.

Since backward SDEs on an infinite time horizon are well defined
dynamic systems (see Theorem \ref{Sec2_THM}), it is natural and
appealing to study the corresponding optimal control and game
problems arising from various fields. For example, in mathematical
finance, the first process $y$ of solution to some backward SDE is
used to represent the price of some European contingent claim, and
the other processes $(z,r)$ of solution are used to characterize the
corresponding portfolio. Then, in an incomplete security market, the
minimum price of some contingent claim can be given by
$\inf_{v(\cdot)\in\mathcal V} y^v(0)$, where $(y^v,z^v,r^v)$ is the
solution of some controlled backward SDEs and $v(\cdot)\in\mathcal
V$ is the related control process.

As an application of theoretical results, we study a kind of
backward stochastic linear-quadratic (LQ) optimal control problems,
and then the general differential game problems. The LQ problems
constitute an extremely important category of optimization problems,
because many problems arising from practice can be modeled by them,
and more importantly, many non-LQ problems can be approximated
reasonably by LQ problems. On the other hand, LQ problems tend to
have elegant and complete solutions due to their simple and nice
structures, which also provide some understanding and preliminaries
for the general nonlinear problems. By virtue of the unique
solvability of forward-backward SDEs, we obtain unique optimal
controls for control problems and unique Nash equilibrium points for
game problems in closed forms. To our best knowledge, it is the
first time to study this kind of infinite horizon backward LQ
problems. The theoretical results of forward-backward SDEs can also
be applied to nonlinear infinite horizon backward optimization
problems. This subject will be detailed in our future works.

The present paper has the following improvements. (i) Compared with
Peng and Shi \cite{PengShi2000}, we clarify many ambiguous
arguments, supplement some necessary details and improve some
proofs. (ii) A general case in which the two initial values are in a
coupled form is studied in this paper, and to deal with it we
introduce a new technique which also can be applied to analyze other
problems. (iii) We provide an important application of the
theoretical results to infinite horizon backward LQ problems. (iv)
In order to match practical problems more accurately, we adopt a
wider jump-diffusion model.

The rest of this paper is organized as follows. In Section
\ref{Sec2}, we introduce some notations and some preliminary results
on infinite horizon (forward) SDEs and backward SDEs, especially an
existence and uniqueness result of backward SDEs. In Section
\ref{Sec3}, we devote ourselves to investigating the infinite
horizon jump-diffusion coupled forward-backward SDEs. We establish
an existence and uniqueness theorem, two stability results and a
comparison theorem for solutions to forward-backward SDEs. In
Section \ref{Sec4}, we apply the existence and uniqueness theorem to
study a kind of infinite horizon backward stochastic LQ optimal
control and differential game problems. We obtain the unique optimal
control for the control problem, and the unique Nash equilibrium
point for the game problem in closed forms.

\section{Notations and preliminaries on SDEs and backward
SDEs}\label{Sec2}

Let $\mathbb R^n$ be the $n$-dimensional Euclidean space with the
usual Euclidean norm $|\cdot|$ and the usual Euclidean inner product
$\langle \cdot,\ \cdot \rangle$. Let $\mathbb R^{n\times m}$ be the
space consisting of all $(n\times m)$ matrices with the inner
product:
\[
\langle A,\ B \rangle = \tr\{ AB^\top \},\quad \mbox{for any } A, B
\in \mathbb R^{n\times m},
\]
where $\top$ denotes the transpose of matrices. Thus the norm $|A|$
of $A$ induced by the inner product is given by $|A| =
\sqrt{\mbox{tr}{AA^\top}}$.
%Let $\mathbb S^n \subset \mathbb
%R^{n\times n}$ be the set of all $(n\times n)$ symmetric matrices,
%$\mathbb S^n_{+}$ (resp. $\mathbb S^n_{-}$) the set of $(n\times n)$
%positive (resp. negative) semi-definite matrices, and $\hat{\mathbb
%S}^n_{+}$ (resp. $\hat{\mathbb S}^n_{-}$) the set of $(n\times n)$
%positive (resp. negative) definite matrices.

Let $(\Omega,\mathcal F,\mathbb F,\mathbb P)$ be a complete filtered
probability space. The filtration $\mathbb F = \{ \mathcal F_t;\
0\leq t<\infty \}$ is generated by two mutually independent
stochastic sources augmented by all $\mathbb P$-null sets. One is a
$d$-dimensional standard Brownian motion $W =
(W_1,W_2,\dots,W_d)^\top$, and the other one consists of $l$
independent Poisson random measures $N = (N_1,N_2,\dots,N_l)^\top$
defined on $\mathbb R_{+}\times\mathcal E$, where $\mathcal E
\subset\mathbb R^{\bar l}\setminus \{0\}$ is a nonempty Borel subset
of some Euclidean space. The compensators of $N$ are $\bar N(dt,de)
= (\pi_1(de)dt, \pi_2(de)dt,\dots,\pi_l(de)dt)$ which make $\{
\tilde N((0,t]\times A) =(N-\bar N)((0,t]\times A);\ 0\leq t<\infty
\}$ a martingale for any $A$ belonging to the Borel field $\mathcal
B(\mathcal E)$ with $\pi_i(A)<\infty$, $i=1,2,\dots,l$. Here, for
each $i=1,2,\dots,l$, $\pi_i$ is a given $\sigma$-finite measure on
the measurable space $(\mathcal E, \mathcal B(\mathcal E))$
satisfying $\int_{\mathcal E} (1\wedge |e|^2) \pi_i(de)<\infty$.

We introduce some spaces:
\begin{itemize}
\item $L^{2,K}_{\mathbb F}(0,\infty;\mathbb R^n)$ where $K\in\mathbb R$, the
space of $\mathbb R^n$-valued $\mathbb F$-progressively measurable
processes $f$ defined on $[0,\infty)$ such that
\[
\norm{f(\cdot)}_{L^{2,K}_{\mathbb F}} := \left(\mathbb
E\int_0^\infty |f(t)|^2 e^{Kt} dt\right)^{\frac 1 2} < \infty,
\]
and for simplicity we denote $L^2_{\mathbb F}(0,\infty;\mathbb R^n)
:= L^{2,0}_{\mathbb F}(0,\infty;\mathbb R^n)$;
\item $S^2_{\mathbb F}(0,T;\mathbb R^n)$ where $T>0$, the space of $\mathbb
F$-progressively measurable processes $f$ which have
right-continuous paths with left limits such that
\[
\norm{f(\cdot)}_{S^2_{\mathbb F}}:= \left(\mathbb E\left[ \sup_{t\in
[0,T]} |f(t)|^2 \right]\right)^{\frac 1 2 } < \infty,
\]
and $S^{2,loc}_{\mathbb F}(0,\infty;\mathbb R^n) := \bigcap_{T>0}
S^2_{\mathbb F}(0,T;\mathbb R^n)$;
\item $\mathcal X^K(0,\infty;\mathbb R^n) := L^{2,K}_{\mathbb F}(0,\infty;\mathbb R^n)\cap S^{2,loc}_{\mathbb
F}(0,\infty;\mathbb R^n)$; similarly, we denote $\mathcal
X(0,\infty;\mathbb R^n) := \mathcal X^0(0,\infty;\mathbb R^n) =
L^2_{\mathbb F}(0,\infty;\mathbb R^n) \cap S^{2,loc}_{\mathbb
F}(0,\infty;\mathbb R^n)$;
\item $L^2(\mathcal E,\mathcal B(\mathcal E),\pi;\mathbb R^{n\times
l})$, the space of $\pi$-almost sure equivalence classes $r(\cdot) =
(r_1(\cdot),\cdots,r_l(\cdot))$ formed by the mappings from
$\mathcal E$ to the space of $\mathbb R^{n\times l}$-valued matrices
such that
\[
\norm{r(\cdot)} := \left( \int_{\mathcal E} \tr \left\{
r(e)\diag(\pi(de))r(e)^\top \right\} \right)^{\frac 1 2} =\left(
 \sum_{i=1}^l \int_{\mathcal E} |r_i(e)|^2 \pi_i(de) \right)^{\frac 1 2}
 <\infty.
\]
This space is equipped with the following inner product:
\[
\langle r(\cdot),\ \bar r(\cdot) \rangle := \int_{\mathcal E} \tr
\left\{ r(e) \diag(\pi(de)) \bar r(e)^\top \right\},\quad\forall\
r(\cdot),\bar r(\cdot)\in L^2(\mathcal E,\mathcal B(\mathcal
E),\pi;\mathbb R^{n\times l});
\]
\item $M^{2,K}_{\mathbb F}(0,\infty;\mathbb R^{n\times l})$ where $K\in\mathbb R$, the space of $\mathbb R^{n\times
l}$-valued, $\mathcal P\otimes \mathcal B(\mathcal E)$-measurable
processes $r$ such that
\[
\begin{aligned}
\norm{r(\cdot,\cdot)}_{M^{2,K}_{\mathbb F}}:=\ & \left(\mathbb
E\int_0^\infty\int_{\mathcal E} \tr \left\{
r(t,e)\diag(\pi(de))r(t,e)^\top \right\} e^{Kt} dt\right)^{\frac 1 2
}\\
=\ & \left( \mathbb E\int_0^\infty \norm{r(t,\cdot)}^2 e^{Kt} dt
\right)^{\frac 1 2} <\infty,
\end{aligned}
\]
where $\mathcal P$ is the $\sigma$-algebra generated by the $\mathbb
F$-progressively measurable processes on $[0,\infty)\times \Omega$,
and we denote $M^{2}_{\mathbb F}(0,\infty;\mathbb R^{n\times l}) : =
M^{2,0}_{\mathbb F}(0,\infty;\mathbb R^{n\times l})$;
\end{itemize}
Clearly, for any $K_1 < K_2$, $L^{2,K_2}_{\mathbb
F}(0,\infty;\mathbb R^n)\subset L^{2,K_1}_{\mathbb
F}(0,\infty;\mathbb R^n)$ and $M^{2,K_2}_{\mathbb
F}(0,\infty;\mathbb R^{n\times l}) \subset M^{2,K_1}_{\mathbb
F}(0,\infty;\mathbb R^{n\times l})$, i.e. the sequences of spaces
$\{L^{2,K}_{\mathbb F}(0,\infty;\mathbb R^n)\}_{K\in\mathbb R}$ and
$\{M^{2,K}_{\mathbb F}(0,\infty;\mathbb R^{n\times
l})\}_{K\in\mathbb R}$ are decreasing in $K$.

Further, we define the space $\mathcal R := \mathbb R^n\times
\mathbb R^n\times\mathbb R^{n\times d} \times L^2(\mathcal
E,\mathcal B(\mathcal E),\pi;\mathbb R^{n\times l})$. For any
$\theta_1 = (x_1,y_1,z_1,r_1(\cdot))$, $\theta_2 =
(x_2,y_2,z_2,r_2(\cdot)) \in \mathcal R$, the inner product is
defined by
\[
\langle \theta_1,\ \theta_2 \rangle := \langle x_1,\ x_2 \rangle
+\langle y_1,\ y_2 \rangle +\langle z_1,\ z_2 \rangle +\langle
r_1(\cdot),\ r_2(\cdot)\rangle.
\]
Then the norm of $\mathcal R$ is deduced by $|\theta| :=
\sqrt{\langle \theta,\ \theta \rangle}$. We also define
\[
\mathcal L^{2,K}_{\mathbb F}(0,\infty) := L^{2,K}_{\mathbb
F}(0,\infty;\mathbb R^n) \times L^{2,K}_{\mathbb F}(0,\infty;\mathbb
R^n) \times L^{2,K}_{\mathbb F}(0,\infty;\mathbb R^{n\times d})
\times M^{2,K}_{\mathbb F}(0,\infty;\mathbb R^{n\times l})
\]
with the norm
\[
\norm{\theta(\cdot)}_{\mathcal L^{2,K}_\mathbb F} = \left\{ \mathbb
E\int_0^\infty |\theta(t)|^2 e^{Kt} dt \right\}^{\frac 1 2} =\left\{
\mathbb E\int_0^\infty \Big[ |x(t)|^2 +|y(t)|^2 +|z(t)|^2
+\norm{r(t,\cdot)}^2 \Big] e^{Kt} dt \right\}^{\frac 1 2}.
\]
Moreover, $\mathcal L^{2}_{\mathbb F}(0,\infty) := \mathcal
L^{2,0}_{\mathbb F}(0,\infty)$.

Now, let us consider an infinite horizon (forward) stochastic
differential equation (SDE):
\[
\begin{aligned}
x(t)=\ & x_0 +\int_0^t b(s,x(s))ds +\int_0^t \sigma(s,x(s)) dW(s)\\
& +\int_0^t\int_{\mathcal E} \gamma(s,e,x(s-)) \tilde N(ds,de),
\quad t\in [0,\infty),
\end{aligned}
\]
which is also expressed in a differential form:
\begin{equation}\label{Sec2_SDE}
\left\{
\begin{aligned}
dx(t) =\ & b(t,x(t)) dt +\sigma(t,x(t)) dW(t) +\int_{\mathcal E}
\gamma(t,e,x(t-)) \tilde N(dt,de),\quad t\in
[0,\infty),\\
x(0) =\ & x_0,
\end{aligned}
\right.
\end{equation}
where $x_0\in\mathbb R^n$, $b: \Omega\times [0,\infty) \times\mathbb
R^n\rightarrow\mathbb R^n$, $\sigma: \Omega\times [0,\infty)
\times\mathbb R^n\rightarrow\mathbb R^{n\times d}$ and
$\gamma:\Omega\times [0,\infty) \times\mathcal E\times\mathbb R^n
\rightarrow \mathbb R^{n\times l}$. Moreover, we introduce the
following assumptions:
\begin{enumerate}[(\mbox{H1}.1)]
\item For any $x\in\mathbb R^n$, $b(\cdot,x)$, $\sigma(\cdot,x)$ are
$\mathbb F$-progressively measurable and $\gamma(\cdot,\cdot,x)$ is
$\mathcal P\otimes\mathcal B(\mathcal E)$-measurable. Moreover,
there exists a constant $K\in\mathbb R$ such that $b(\cdot,0)\in
L^{2,K}_{\mathbb F}(0,\infty;\mathbb R^n)$, $\sigma(\cdot,0)\in
L^{2,K}_{\mathbb F}(0,\infty;\mathbb R^{n\times d})$ and
$\gamma(\cdot,\cdot,0) \in M^{2,K}_{\mathbb F}(0,\infty;\mathbb
R^{n\times l})$.
\item $b$, $\sigma$ and $\gamma$ are Lipschitz continuous with
respect to $x$, i.e. there exists a constant $C>0$ such that for any
$t\in [0,\infty)$, any $x_1$, $x_2\in\mathbb R^n$,
\[
|b(t,x_1)-b(t,x_2)| +|\sigma(t,x_1)-\sigma(t,x_2)|
+\norm{\gamma(t,\cdot,x_1)-\gamma(t,\cdot,x_2)} \leq C|x_1-x_2|.
\]
\end{enumerate}
By the classical theory of SDEs, under Assumptions (H1.1)-(H1.2),
SDE \eqref{Sec2_SDE} admits a unique strong solution. Precisely, for
any $T\in [0,\infty)$,
\[
\mathbb E\left[ \sup_{x\in [0,T]} |x(t)|^2 \right] <\infty,
\]
i.e. $x\in S^2_{\mathbb F}(0,T;\mathbb R^n)$, and then $x\in
S^{2,loc}_{\mathbb F}(0,\infty;\mathbb R^n)$.
\begin{proposition}\label{Sec2_Prop}
Let Assumptions (H1.1)-(H1.2) hold. We further assume the unique
solution $x$ of SDE \eqref{Sec2_SDE} belongs to $L^{2,K}_{\mathbb
F}(0,\infty;\mathbb R^n)$ where the constant $K$ is given by (H1.1).
Then, we have
\begin{enumerate}[(i)]
\item $x\in\mathcal X^K(0,\infty;\mathbb R^n)$;
\item $\mathbb E[|x(t)|^2 e^{Kt}]$ is bounded and continuous;
\item $\lim_{t\rightarrow\infty} \mathbb E [|x(t)|^2 e^{Kt}] =0$.
\end{enumerate}
\end{proposition}

\begin{proof}
The assertion (i) is obvious since $x$ belongs to both
$S^{2,loc}_{\mathbb F} (0,\infty;\mathbb R^n)$ and $L^{2,K}_{\mathbb
F} (0,\infty;\mathbb R^n)$. For any $t\in [0,\infty)$, we apply
It\^{o}'s formula to $|x(s)|^2 e^{Ks}$ on the interval $[0,t]$:
\[
\begin{aligned}
\mathbb E \left[ |x(t)|^2 e^{Kt} \right] =\ & |x_0|^2 +\mathbb
E\int_0^t 2\Big\langle x(s),\ b(s,x(s)) \Big\rangle e^{Ks} ds\\
& +\mathbb E\int_0^t \Big[ K|x(s)|^2 +|\sigma(s,x(s))|^2
+\norm{\gamma(s,\cdot,x(s))}^2 \Big] e^{Ks} ds.
\end{aligned}
\]
By the Lipschitz condition (H1.2), we have
\[
\begin{aligned}
\mathbb E \left[ |x(t)|^2 e^{Kt} \right] \leq\ & |x_0|^2 +
(1+|K|+2C+4C^2) \mathbb E \int_0^t   |x(s)|^2
e^{Ks} ds\\
& +\mathbb E\int_0^t \left[ |b(s,0)|^2 +2|\sigma(s,0)|^2
+2\norm{\gamma(s,\cdot,0)}^2 \right] e^{Ks} ds,
\end{aligned}
\]
where $C$ is the Lipschitz constant. Due to the fact that
$x(\cdot)$, $b(\cdot,0)\in L^{2,K}_{\mathbb F}(0,\infty;\mathbb
R^n)$, $\sigma(\cdot,0)\in L^{2,K}_{\mathbb F} (0,\infty;\mathbb
R^{n\times d})$ and $\gamma(\cdot,\cdot,0)\in M^{2,K}_{\mathbb
F}(0,\infty;\mathbb R^{n\times l})$ (see (H1.1)), the above
inequality implies the deterministic process $\{\mathbb E[|x(t)|^2
e^{Kt}];\ t\geq 0\}$ is bounded. Moreover, in the same way, applying
It\^{o}'s formula to $|x(s)|^2 e^{Ks}$ on the interval $[t_1,t_2]$
leads to
\[
\begin{aligned}
\Big| \mathbb E\left[ |x(t_2)|^2 e^{Kt_2}\right] -\mathbb E \left[
|x(t_1)|^2 e^{Kt_1} \right] \Big| \leq\ & (1+|K|+2C+4C^2) \mathbb E
\int_{t_1}^{t_2} |x(s)|^2
e^{Ks} ds\\
& +\mathbb E\int_{t_1}^{t_2} \left[ |b(s,0)|^2 +2|\sigma(s,0)|^2
+2\norm{\gamma(s,\cdot,0)}^2 \right] e^{Ks} ds,
\end{aligned}
\]
which implies that the process $\{\mathbb E[|x(t)|^2 e^{Kt}];\ t\geq
0\}$ is continuous. We have proved (ii). The above inequality also
shows that, $\mathbb E\left[ |x(t_2)|^2 e^{Kt_2}\right] -\mathbb E
\left[ |x(t_1)|^2 e^{Kt_1} \right] \rightarrow 0$ as
$t_1,t_2\rightarrow\infty$, then $\lim_{t\rightarrow\infty} \mathbb
E\left[ |x(t)|^2 e^{Kt} \right]$ exists. Furthermore, due to
\[
\int_0^\infty \mathbb E\left[ |x(t)|^2 e^{Kt} \right] dt <\infty,
\]
we get the desired conclusion
\[
\lim_{t\rightarrow\infty} \mathbb E\left[ |x(t)|^2 e^{Kt} \right]
=0.
\]
\end{proof}

In the rest of this section, we shall consider an infinite horizon
backward SDE as follows:
\[
\begin{aligned}
y(t) =\ & \int_t^\infty \Big[ G(s,y(s),z(s),r(s,\cdot)) +\varphi(s)
\Big] ds -\int_t^\infty z(s) dW(s)\\
& -\int_t^\infty\int_{\mathcal E} r(s,e) \tilde N(ds,de), \quad t\in
[0,\infty),
\end{aligned}
\]
which is denoted also in the following differential form:
\begin{equation}\label{Sec2_BSDE}
-dy(t) = \Big[ G(t,y(t),z(t),r(t,\cdot)) +\varphi(t) \Big]dt
-z(t)dW(t) -\int_{\mathcal E} r(t,e) \tilde N(dt, de),\quad t\in
[0,\infty),
\end{equation}
where $G: \Omega\times [0,\infty) \times \mathbb R^m \times \mathbb
R^{m\times d} \times L^2(\mathcal E,\mathcal B(\mathcal
E),\pi;\mathbb R^{m\times l}) \rightarrow \mathbb R^m$ and $\varphi:
\Omega\times [0,\infty) \rightarrow \mathbb R^m$. We assume the
following assumptions on the coefficients $(G,\varphi)$:
\begin{enumerate}[(\mbox{H2}.1)]
\item For any $(y,z,r(\cdot))\in \mathbb R^m \times \mathbb R^{m\times d} \times L^2(\mathcal E, \mathcal B(\mathcal E), \pi; \mathbb R^{m\times
l})$, $G(\cdot,y,z,r(\cdot))$ is $\mathbb F$-progressively
measurable, and satisfies $G(t,0,0,0) = 0$ for any $t\in
[0,\infty)$. Moreover, there exists a constant $K\in\mathbb R$ such
that $\varphi\in L^{2,K}_{\mathbb F}(0,\infty;\mathbb R^m)$.
\item $G$ is Lipschitz continuous with respect to $(y,z,r(\cdot))$,
i.e. there exist constants $C_0\geq 0$, $C_1\geq 0$ and $C_2\geq 0$
such that for any $t\in [0,\infty)$, any $y_1, y_2 \in\mathbb R^m$,
any $z_1, z_2 \in\mathbb R^{m\times d}$, any $r_1(\cdot), r_2(\cdot)
\in L^2(\mathcal E,\mathcal B(\mathcal E),\pi;\mathbb R^{m\times
l})$,
\[
|G(t,y_1,z_1,r_1(\cdot)) -G(t,y_2,z_2,r_2(\cdot))| \leq C_0
|y_1-y_2| +C_1|z_1-z_2| +C_2\norm{r_1(\cdot)-r_2(\cdot)}.
\]
\item $G$ satisfies some `weak monotonicity' conditions in the
following sense: there exists a constant $\rho\in\mathbb R$ such
that for any $t\in [0,\infty)$, any $y_1,y_2\in\mathbb R^m$, any
$z\in\mathbb R^{m\times d}$, any $r(\cdot)\in L^2(\mathcal
E,\mathcal B(\mathcal E),\pi;\mathbb R^{m\times l})$,
\[
\langle G(t,y_1,z,r(\cdot))-G(t,y_2,z,r(\cdot)), y_1-y_2 \rangle
\leq -\rho |y_1-y_2|^2.
\]
\item $K+2\rho-2C_1^2-2C_2^2 >0$.
\end{enumerate}

\begin{remark}
\begin{enumerate}[(i)]
\item In the classical form of backward SDEs, the coefficient is
usually denoted by $g(t,y,z,r(\cdot))$. While in the present paper,
for convenience, we set
\[
G(t,y,z,r(\cdot)) = g(t,y,z,r(\cdot))-g(t,0,0,0),\quad \varphi(t) =
g(t,0,0,0),
\]
and denote the coefficient $g$ by $G+\varphi$.
\item From the density of real numbers, it is easy to see that (H2.4) is
equivalent to the following statement: (H2.4') There exists a
constant $\delta>0$ such that
\begin{equation}\label{Sec2_delta}
K+2\rho-2C_1^2-2C_2^2-\delta>0.
\end{equation}
%\item When they studied infinite horizon backward SDEs driven by Brownian motions, Peng
%and Shi \cite{PengShi2000} introduced some assumptions similar to
%(H2.3) and (H2.4). However, in \cite{PengShi2000}, the constant
%$\rho$ in (H2.3) was assumed to be positive and (H2.4) was stated in
%the form of (H2.4').
\end{enumerate}
\end{remark}

Naturally, a triple of mappings $(y(\cdot),z(\cdot),r(\cdot,\cdot))$
is called an adapted solution to backward SDE \eqref{Sec2_BSDE} if
and only if $y$ is an $\mathbb R^m$-valued $\mathbb F$-progressively
measurable process, $z$ is an $\mathbb R^{m\times d}$-valued
$\mathbb F$-progressively measurable process, $r$ is an $\mathbb
R^{m\times l}$-valued $\mathcal P\otimes\mathcal B(\mathcal
E)$-measurable process, and $(y,z,r)$ satisfies \eqref{Sec2_BSDE}.
Similar to forward SDEs, we have the following

\begin{corollary}\label{Sec2_Corollary}
Let Assumptions (H2.1) and (H2.2) hold. We further assume
$(y,z,r)\in L^{2,K}_{\mathbb F} (0,\infty;\mathbb R^m) \times
L^{2,K}_{\mathbb F}(0,\infty;\mathbb R^{m\times d}) \times
M^{2,K}_{\mathbb F}(0,\infty;\mathbb R^{m\times l})$ is a solution
to backward SDE \eqref{Sec2_BSDE} where the constant $K$ is given by
(H2.1). Then we have
\begin{enumerate}[(i)]
\item $y\in\mathcal X^K(0,\infty;\mathbb R^m)$;
\item $\mathbb E[|y(t)|^2 e^{Kt}]$ is bounded and continuous;
\item $\lim_{t\rightarrow\infty} \mathbb E [|y(t)|^2 e^{Kt}] =0$.
\end{enumerate}
\end{corollary}

\begin{proof}
Since $(y,z,r)$ is a solution to backward SDE \eqref{Sec2_BSDE},
then for any $t\in [0,\infty)$,
\[
\begin{aligned}
y(t) =\ & \int_t^\infty \Big[ G(s,y(s),z(s),r(s,\cdot)) +\varphi(s)
\Big] ds -\int_t^\infty z(s) dW(s) -\int_t^\infty\int_{\mathcal E}
r(s,e) \tilde N(ds,de)\\
=\ & \int_0^\infty \Big[ G(s,y(s),z(s),r(s,\cdot)) +\varphi(s) \Big]
ds -\int_0^\infty z(s) dW(s) -\int_0^\infty\int_{\mathcal E} r(s,e)
\tilde N(ds,de)\\
& -\int_0^t \Big[ G(s,y(s),z(s),r(s,\cdot)) +\varphi(s) \Big] ds
+\int_0^t z(s) dW(s) +\int_0^t\int_{\mathcal E} r(s,e) \tilde
N(ds,de)\\
=\ & y(0) -\int_0^t \Big[ G(s,y(s),z(s),r(s,\cdot)) +\varphi(s)
\Big] ds +\int_0^t z(s) dW(s) +\int_0^t\int_{\mathcal E} r(s,e)
\tilde N(ds,de).
\end{aligned}
\]
So the process $y\in L^{2,K}_{\mathbb F}(0,\infty;\mathbb R^m)$ can
be regarded as an adapted solution to a forward SDE with
\[
b(t,y) = -\Big[ G(t,y,z(t),r(t,\cdot)) +\varphi(t) \Big],\quad
\sigma(t,y) = z(t),\quad \gamma(t,e,y) =r(t,e).
\]
Under Assumptions (H2.1) and (H2.2), it is easy to check that the
above coefficients $(b,\sigma,\gamma)$ satisfy Assumptions
(H1.1)-(H1.2). By Proposition \ref{Sec2_Prop}, we get the
conclusions.
\end{proof}

In order to obtain an existence and uniqueness result for backward
SDE \eqref{Sec2_BSDE}, we first establish the following a priori
estimate.

\begin{lemma}\label{Sec2_Lemma}
Let Assumptions (H2.1)-(H2.4) hold. Let $(y_1,z_1,r_1)$ and
$(y_2,z_2,r_2)\in L^{2,K}_{\mathbb F}(0,\infty;\mathbb R^m) \times
L^{2,K}_{\mathbb F}(0,\infty;\mathbb R^{m\times d}) \times
M^{2,K}_{\mathbb F}(0,\infty;\mathbb R^{m\times l})$ be solutions to
backward SDEs \eqref{Sec2_BSDE} with $\varphi=\varphi_1$ and
$\varphi=\varphi_2$ respectively. Then we have
\begin{equation}\label{Sec2_A_Priori}
\begin{aligned}
& \mathbb E\int_0^\infty \Big[ (K+2\rho-2C_1^2 -2C_2^2 -\delta)
|y_1(t)-y_2(t)|^2\\
& +\frac 1 2 |z_1(t)-z_2(t)|^2 +\frac 1 2 \norm{r_1(t,\cdot)-r_2(t,\cdot)}^2 \Big] e^{Kt} dt\\
\leq\ & \frac 1 \delta \mathbb E\int_0^\infty
|\varphi_1(t)-\varphi_2(t)|^2 e^{Kt} dt,
\end{aligned}
\end{equation}
where $\delta>0$ is defined in \eqref{Sec2_delta}.
\end{lemma}

\begin{proof}
We denote
\[
\begin{aligned}
& \hat \varphi(t) := \varphi_1(t)-\varphi_2(t), \quad \hat y(t) :=
y_1(t)-y_2(t),\\
& \hat z(t) := z_1(t)-z_2(t), \quad \hat r(t,e) :=
r_1(t,e)-r_2(t,e),
\end{aligned}
\]
for any $(\omega,t,e)\in \Omega \times [0,\infty)\times\mathcal E$.
For any given $T>0$, we apply It\^{o}'s formula to $|\hat y(t)|^2
e^{Kt}$ on the interval $[0,T]$:
\[
\begin{aligned}
& \mathbb E\Big[ |\hat y(T)|^2 e^{KT} \Big] -|\hat y(0)|^2\\
=\ & \mathbb E\int_0^T \left[ K|\hat y(t)|^2 +|\hat z(t)|^2
+\norm{\hat r(t,\cdot)}^2 \right]
e^{Kt} dt\\
& -2 \mathbb E\int_0^T \Big\langle \hat y(t),\
G(t,y_1(t),z_1(t),r_1(t,\cdot)) -G(t,y_2(t),z_2(t),r_2(t,\cdot))
+\hat\varphi(t) \Big\rangle e^{Kt} dt.
\end{aligned}
\]
Then
\[
\begin{aligned}
& |\hat y(0)|^2 +\mathbb E\int_0^T \left[ K|\hat y(t)|^2 +|\hat
z(t)|^2 +\norm{\hat r(t,\cdot)}^2 \right] e^{Kt} dt\\
=\ & \mathbb E \Big[ |\hat y(T)|^2 e^{KT} \Big] +2\mathbb E\int_0^T
\left\langle \hat y(t),\ \hat\varphi(t) \right\rangle e^{Kt} dt\\
& +2\mathbb E\int_0^T \Big\langle \hat y(t),\
G(t,y_1(t),z_1(t),r_1(t,\cdot)) -G(t,y_2(t),z_1(t),r_1(t,\cdot))
\Big\rangle e^{Kt} dt\\
& +2\mathbb E\int_0^T \Big\langle \hat y(t),\
G(t,y_2(t),z_1(t),r_1(t,\cdot)) -G(t,y_2(t),z_2(t),r_2(t,\cdot))
\Big\rangle e^{Kt} dt.
\end{aligned}
\]
By Assumptions (H2.2) and (H2.3), we have
\[
\begin{aligned}
& |\hat y(0)|^2 +\mathbb E\int_0^T \left[ K|\hat y(t)|^2 +|\hat
z(t)|^2 +\norm{\hat r(t,\cdot)}^2 \right] e^{Kt} dt\\
\leq\ & \mathbb E \Big[ |\hat y(T)|^2 e^{KT} \Big] +\mathbb
E\int_0^T \Big[ 2|\hat y(t)| |\hat\varphi(t)| -2\rho |\hat y(t)|^2
+2|\hat y(t)|(C_1|\hat z(t)| +C_2\norm{\hat r(t,\cdot)}) \Big]
e^{Kt} dt\\
\leq\ & \mathbb E \Big[ |\hat y(T)|^2 e^{KT} \Big] +\mathbb
E\int_0^T  \Big[ (\delta-2\rho+2C_1^2+2C_2^2) |\hat y(t)|^2 +\frac 1
2 |\hat z(t)|^2 +\frac 1 2 \norm{\hat r(t,\cdot)}^2
+\frac{1}{\delta} |\hat\varphi(t)|^2 \Big] e^{Kt} dt.
\end{aligned}
\]
Therefore,
\[
\begin{aligned}
& |\hat y(0)|^2 +\mathbb E\int_0^T
\Big[(K+2\rho-2C_1^2-2C_2^2-\delta) |\hat y(t)|^2 +\frac 1 2 |\hat
z(t)|^2 +\frac 1 2 \norm{\hat r(t,\cdot)}^2\Big] e^{Kt} dt\\
\leq\ & \mathbb E \Big[ |\hat y(T)|^2 e^{KT} \Big] +\frac{1}{\delta}
\mathbb E\int_0^T |\hat\varphi(t)|^2 e^{Kt} dt.
\end{aligned}
\]
Let $T\rightarrow\infty$. Thanks to Corollary \ref{Sec2_Corollary},
we have
\[
\begin{aligned}
& |\hat y(0)|^2 +\mathbb E\int_0^\infty
\Big[(K+2\rho-2C_1^2-2C_2^2-\delta) |\hat y(t)|^2 +\frac 1 2 |\hat
z(t)|^2 +\frac 1 2 \norm{\hat r(t,\cdot)}^2\Big] e^{Kt} dt\\
\leq\ & \frac{1}{\delta} \mathbb E\int_0^\infty |\hat\varphi(t)|^2
e^{Kt} dt,
\end{aligned}
\]
which implies \eqref{Sec2_A_Priori}.
\end{proof}

\begin{theorem}\label{Sec2_THM}
Let Assumptions (H2.1)-(H2.4) hold. When $K>0$, the backward SDE
\eqref{Sec2_BSDE} admits a unique solution $(y,z,r)\in
L^{2,K}_{\mathbb F}(0,\infty;\mathbb R^m) \times L^{2,K}_{\mathbb
F}(0,\infty;\mathbb R^{m\times d}) \times M^{2,K}_{\mathbb
F}(0,\infty;\mathbb R^{m\times l})$.
\end{theorem}

\begin{proof}
Clearly, the a priori estimate \eqref{Sec2_A_Priori} implies the
uniqueness. For the existence, we employ the method used in
\cite{PengShi2000} to construct an adapted solution. In detail, for
$n=1,2,\dots$, we define
\[
\varphi_n(t) = \mathbbm{1}_{[0,n]}(t) \varphi(t),\quad t\in
[0,\infty).
\]
Obviously, the sequence $\{\varphi_n\}_{n=1}^{\infty}$ converges to
$\varphi$ in $L^{2,K}_{\mathbb F}(0,\infty;\mathbb R^m)$. For each
$n$, let $(\bar y_n,\bar z_n,\bar r_n)$ be the unique adapted
solution of the following finite horizon backward SDE:
\[
\begin{aligned}
\bar y_n(t) =\ & \int_t^n \Big[ G(s,\bar y_n(s),\bar z_n(s),\bar
r_n(s,\cdot)) +\varphi_n(s) \Big] ds -\int_t^n \bar z_n(s) dW(s)\\
& -\int_t^n \int_{\mathcal E} \bar r_n(s,e) \tilde N(ds,de),\quad
t\in [0,n].
\end{aligned}
\]
Furthermore, we define
\[
(y_n(t),z_n(t),r_n(t,\cdot)) := \left\{
\begin{aligned}
& (\bar y_n(t), \bar z_n(t), \bar r_n(t,\cdot)), \quad && t\in [0,n],\\
& (0,0,0),\quad && t\in (n,\infty).
\end{aligned}
\right.
\]
Obviously, $(y_n,z_n,r_n)\in L^{2,K}_{\mathbb F}(0,\infty;\mathbb
R^m) \times L^{2,K}_{\mathbb F}(0,\infty;\mathbb R^{m\times d})
\times M^{2,K}_{\mathbb F}(0,\infty;\mathbb R^{m\times l})$. Since
$G(s,0,0,0)=0$ (see Assumption (H2.1)), then $(y_n,z_n,r_n)$ solves
the following infinite horizon backward SDE:
\[
\begin{aligned}
y_n(t) =\ & \int_t^\infty \Big[ G(s, y_n(s), z_n(s),
r_n(s,\cdot)) +\varphi_n(s) \Big] ds -\int_t^\infty z_n(s) dW(s)\\
& -\int_t^\infty \int_{\mathcal E} r_n(s,e) \tilde N(ds,de),\quad
t\in [0,\infty).
\end{aligned}
\]
Lemma \ref{Sec2_Lemma} implies that $\{(y_n,z_n,r_n)\}_{n=1}^\infty$
is a Cauchy sequence in $L^{2,K}_{\mathbb F}(0,\infty;\mathbb R^m)
\times L^{2,K}_{\mathbb F}(0,\infty;\mathbb R^{m\times d}) \times
M^{2,K}_{\mathbb F}(0,\infty;\mathbb R^{m\times l})$. We denote by
$(y,z,r)$ the limit of $\{(y_n,z_n,r_n)\}_{n=1}^\infty$, and shall
show that $(y,z,r)$ solves the backward SDE \eqref{Sec2_BSDE}.

First, when $K>0$ we deduce that
\[
\begin{aligned}
\mathbb E\left[\int_t^\infty (z_n(s)-z(s)) dW(s)\right]^2 =\ &
\mathbb E\int_t^\infty |z_n(s)-z(s)|^2 ds\\
\leq\ & \mathbb E \int_0^\infty |z_n(s)-z(s)|^2 e^{Ks} ds
\rightarrow 0,\quad \mbox{as } n\rightarrow\infty,
\end{aligned}
\]
i.e. the item $\int_t^\infty z_n(s) dW(s)$ converges to
$\int_t^\infty z(s) dW(s)$ in $L^2(\Omega,\mathcal F,\mathbb
P;\mathbb R^m)$ which is the space of $\mathcal F$-measurable square
integrable random variables. The same argument also leads to a
similar conclusion: the item $\int_t^\infty \int_{\mathcal E}
r_n(s,e) \tilde N(ds, de)$ converges to $\int_t^\infty\int_{\mathcal
E} r(s,e) \tilde N(ds, de)$ in $L^2(\Omega,\mathcal F,\mathbb
P;\mathbb R^m)$. Second, for any $K>0$, we have
\[
\begin{aligned}
& \mathbb E\left[ \int_t^\infty \Big(
G(s,y_n(s),z_n(s),r_n(s,\cdot)) -G(s,y(s),z(s),r(s,\cdot)) \Big) ds
\right]^2\\
\leq\ & \mathbb E\left[ \int_0^\infty \Big|
G(s,y_n(s),z_n(s),r_n(s,\cdot)) -G(s,y(s),z(s),r(s,\cdot)) \Big|
e^{\frac K 2 s} e^{-\frac K 2 s} ds \right]^2\\
\leq\ & \mathbb E\left[ \left( \int_0^\infty \Big|
G(s,y_n(s),z_n(s),r_n(s,\cdot)) -G(s,y(s),z(s),r(s,\cdot)) \Big|^2
e^{Ks} ds \right) \left(\int_0^\infty e^{-Ks} ds \right) \right]\\
\leq\ & C \mathbb E\int_0^\infty \Big[ |y_n(s)-y(s)|^2
+|z_n(s)-z(s)|^2 +\norm{r_n(s,\cdot)-r(s,\cdot)}^2 \Big] e^{Ks} ds\\
& \rightarrow 0,\quad \mbox{as } n\rightarrow\infty,
\end{aligned}
\]
i.e. the item $\int_t^\infty G(s,y_n(s),z_n(s),r_n(s,\cdot)) ds$
converges to $\int_t^\infty G(s,y(s),z(s),r(s,\cdot)) ds$ in
$L^2(\Omega,\mathcal F,\mathbb P;\mathbb R^m)$. We notice that, here
in order to dominate the $L^1$-norm by the $L^2$-norm, we have to
restrict $K>0$. This is different from the case of finite time
intervals. The same argument also leads to $\int_t^\infty
\varphi_n(s) ds$ converges to $\int_t^\infty \varphi(s) ds$ in
$L^2(\Omega,\mathcal F,\mathbb P;\mathbb R^m)$. At last, since
$\lim_{n\rightarrow\infty} \mathbb E \int_0^\infty |y_n(t)-y(t)|^2
e^{Kt} dt = 0$, there exists a subsequence of $\{y_n\}$ such that
\[
\lim_{n\rightarrow\infty} \mathbb E\Big[ |y_n(t)-y(t)|^2 \Big] = 0,
\quad \mbox{for almost everywhere } t \in [0,\infty).
\]
The proof is completed.
\end{proof}

\section{Coupled forward-backward SDEs}\label{Sec3}

In this section, we study the following kind of coupled
forward-backward SDEs driven by both Brownian motions and Poisson
processes on the infinite interval $[0,\infty)$:
\begin{equation}\label{Sec3_FBSDE}
\left\{
\begin{aligned}
dx(t) =\ & b(t,x(t),y(t),z(t),r(t,\cdot)) dt
+\sigma(t,x(t),y(t),z(t),r(t,\cdot)) dW(t)\\
& +\int_{\mathcal E} \gamma(t,e,x(t-),y(t-),z(t),r(t,e)) \tilde
N(dt,de),\\
-dy(t) =\ & g(t,x(t),y(t),z(t),r(t,\cdot)) dt -z(t) dW(t)
-\int_{\mathcal E} r(t,e) \tilde N(dt, de),\\
x(0) =\ & \Phi(y(0)),
\end{aligned}
\right.
\end{equation}
where $\Phi: \mathbb R^n\rightarrow \mathbb R^n$, $(b,\sigma,g):
\Omega\times [0,\infty)\times \mathcal R \rightarrow \mathbb R^n
\times \mathbb R^{n\times d} \times \mathbb R^n$ and $\gamma:
\Omega\times [0,\infty)\times\mathcal E\times \mathcal R \rightarrow
\mathbb R^{n\times l}$. Similar to Hu and Peng \cite{HuPeng1995},
for any $\theta = (x,y,z,r(\cdot)) \in\mathcal R$, we use the
notation $A(t,\theta) := (-g(t,\theta), b(t,\theta),
\sigma(t,\theta), \gamma(t,\cdot,\theta))$. Now we give the
following assumptions:
\begin{enumerate}[(\mbox{H3}.1)]
\item For any $\theta\in \mathcal R$, $b(\cdot,\theta)$, $\sigma(\cdot,\theta)$, $g(\cdot,\theta)$ are
$\mathbb F$-progressively measurable and
$\gamma(\cdot,\cdot,\theta)$ is $\mathcal P\otimes \mathcal
B(\mathcal E)$-measurable. Moreover, there exists a constant $K>0$
such that $A(\cdot,0)\in \mathcal L^{2,K}_{\mathbb F}(0,\infty)$.
\item $A$ and $\Phi$ are Lipschitz continuous with respect to $\theta$ and $y$ respectively, i.e.
there exists a constant $C>0$ such that for any $t\in [0,\infty)$,
any $\theta_1$, $\theta_2 \in \mathcal R$, any $y_1$, $y_2\in\mathbb
R^n$,
\[
\begin{aligned}
\mbox{(i)} & \quad |A(t,\theta_1)-A(t,\theta_2)| \leq C
|\theta_1-\theta_2|,\\
\mbox{(ii)} & \quad |\Phi(y_1)-\Phi(y_2)| \leq C |y_1-y_2|.
\end{aligned}
\]
\item $A$ and $\Phi$ satisfy the monotonicity conditions in the sense: there exists a constant $\mu>0$ such that for any $t\in
[0,\infty)$, any $\theta_1$, $\theta_2\in\mathcal R$, any $y_1$,
$y_2\in \mathbb R^n$,
\[
\begin{aligned}
\mbox{(i)} & \quad \langle A(t,\theta_1)-A(t,\theta_2),\
\theta_1-\theta_2 \rangle \leq -\mu |\theta_1-\theta_2|^2,\\
\mbox{(ii)} & \quad \langle \Phi(y_1)-\Phi(y_2),\ y_1-y_2 \rangle
\leq 0.
\end{aligned}
\]
\item $2\mu -K >0$.
\end{enumerate}

\begin{remark}\label{Sec3_Remark_Artificial}
Assumption (H3.4) is artificial. In fact, if it does not hold true,
then we can find a $\bar K\in (0,K)$ such that $2\mu-\bar K >0$. Due
to the decreasing property of $\{ L^{2,K}_{\mathbb
F}(0,\infty;\mathbb R^n) \}_{K\in\mathbb R}$ and $\{
M^{2,K}_{\mathbb F}(0,\infty;\mathbb R^{n\times l}) \}_{K\in\mathbb
R}$, Assumption (H3.1) implies $A(\cdot,0)$ also belongs to
$\mathcal L^{2,\bar K}_{\mathbb F}(0,\infty)$. So we can deal with
the corresponding problems in a larger space. However, for
convenience, we would like to keep (H3.4) in this paper.
\end{remark}

Next we employ the method of continuation originally introduced by
Hu and Peng \cite{HuPeng1995} to obtain the existence and uniqueness
of the forward-backward SDE \eqref{Sec3_FBSDE}. For this purpose, we
introduce a family of infinite horizon forward-backward SDEs
parametrized by $\alpha\in [0,1]$:
\begin{equation}\label{Sec3_FBSDE_Param}
\left\{
\begin{aligned}
dx^\alpha(t) =\ & \left[\alpha b(t,\theta^\alpha(t))
-\mu(1-\alpha)y^\alpha(t) +\phi(t)\right]dt\\
& +\left[ \alpha\sigma(t,\theta^\alpha(t)) -\mu(1-\alpha)z^\alpha(t)
+\psi(t) \right] dW(t)\\
& +\int_{\mathcal E} \left[\alpha \gamma(t,e,\theta^\alpha(t-))
-\mu(1-\alpha)r^\alpha(t,e)+\xi(t,e)\right] \tilde N(dt, de),\\
-dy^\alpha(t) =\ & \left[ \alpha g(t,\theta^\alpha(t))
+\mu(1-\alpha) x^\alpha(t) +\eta(t) \right] dt -z^\alpha(t) dW(t)\\
& -\int_{\mathcal E} r^\alpha(t,e) \tilde N(dt, de),\\
x^\alpha(0) =\ & \Phi(y^\alpha(0)),
\end{aligned}
\right.
\end{equation}
where $(\eta,\phi,\psi,\xi) \in \mathcal L^{2,K}_{\mathbb
F}(0,\infty)$ and we denote $\theta^\alpha(t) :=
(x^\alpha(t),y^\alpha(t),z^\alpha(t),r^\alpha(t,\cdot))$,
$\theta^\alpha(t-) :=
(x^\alpha(t-),y^\alpha(t-),z^\alpha(t),r^\alpha(t,e))$. We notice
that, the coefficient $\Phi$ is not parameterized as the same as
other coefficients $(b,\sigma,\gamma,g)$. This is a difference from
the traditional parameterization technique used in
\cite{HuPeng1995,PengShi2000,PengWu1999,Yong1997,Yong2010}.

When $\alpha =0$, the forward-backward SDE \eqref{Sec3_FBSDE_Param}
is reduced to
\begin{equation}\label{Sec3_FBSDE_Param_0}
\left\{
\begin{aligned}
dx^0(t) =\ & \left[ -\mu y^0(t) +\phi(t) \right] dt +\left[ -\mu
z^0(t) +\psi(t) \right] dW(t)\\
& +\int_{\mathcal E} \left[ -\mu r^0(t,e) +\xi(t,e) \right] \tilde
N(dt, de),\\
-dy^0(t) =\ & \left[ \mu x^0(t) +\eta(t) \right]dt -z^0(t) dW(t)
-\int_{\mathcal E} r^0(t,e) \tilde N(dt, de),\\
x^0(0) =\ & \Phi(y^0(0)).
\end{aligned}
\right.
\end{equation}
Before proving the unique solvability result for
\eqref{Sec3_FBSDE_Param_0}, we need to consider an algebraic
equation related to the coupling of initial conditions. The
technique dealing with the algebraic equation is similar to the
proof of Lemma 3.4 in \cite{WuYu2014}.

\begin{lemma}\label{Sec3_Lemma_AE}
Let Assumptions (H3.1)-(H3.4) hold. For any $p\in\mathbb R^n$, the
following algebraic equation
\begin{equation}\label{Sec3_AE}
x = \Phi(x+p)
\end{equation}
admits a unique solution $x\in\mathbb R^n$.
\end{lemma}

\begin{proof}
Firstly, we show the uniqueness. If both $x_1$ and $x_2$ satisfy the
algebraic equation \eqref{Sec3_AE}, then
\[
x_1-x_2 = \Phi(x_1+p) -\Phi(x_2+p).
\]
Making inner product with $x_1-x_2$, by the monotonicity condition
on $\Phi$ (see (H3.3)), we get
\[
|x_1-x_2|^2 = \langle \Phi(x_1+p) -\Phi(x_2+p),\ x_1-x_2 \rangle
\leq 0.
\]
We proved the uniqueness.

Secondly, we prove the existence. We define a new function
\begin{equation}
\lambda(x,p) = \Phi(x+p)-x,\quad (x,p)\in\mathbb R^n\times\mathbb
R^n.
\end{equation}
Making inner product of $\lambda$ and $x$, by the monotonicity
condition of $\Phi$, we get
\[
\begin{aligned}
\langle \lambda(x,p),\ x \rangle =\ & \langle \Phi(x+p),\ x \rangle
-|x|^2\\
=\ & \langle \Phi(x+p)-\Phi(p),\ x \rangle +\langle \Phi(p),\ x
\rangle -|x|^2\\
\leq\ & -|x|^2 +\langle \Phi(p),\ x \rangle.
\end{aligned}
\]
From the inequality: $\langle a,\ b \rangle \leq (1/2)
(|a|^2+|b|^2)$, we have
\begin{equation}\label{Sec3_Temp1_Lemma_AE}
\langle \lambda(x,p),\ x \rangle \leq -\frac 1 2 |x|^2 +\frac 1 2
|\Phi(p)|^2.
\end{equation}
We assert that the above inequality implies that, for any
$p\in\mathbb R^n$, there exists an $x(p)$ such that $\lambda(x(p),p)
=0$. This is equivalent to the existence of the algebraic equation
\eqref{Sec3_AE}. In order to highlight the idea of the proof, here
we only prove this conclusion in a simple case where $n=1$. For the
general case $n>1$, the proof is a bit complicated and technical, we
would like to omit it. The interested readers can be referred to
Appendix in \cite{WuYu2014}. When $n=1$, the inequality
\eqref{Sec3_Temp1_Lemma_AE} is rewritten as
\begin{equation}\label{Sec3_Temp2_Lemma_AE}
\lambda(x,p) x \leq -\frac 1 2 x^2 +\frac{\Phi^2(p)}{2}.
\end{equation}
\begin{enumerate}[(i)]
\item When $x<0$, dividing $x$ on both sides of the inequality
\eqref{Sec3_Temp2_Lemma_AE}, we have
\[
\lambda(x,p) \geq -\frac 1 2 x +\frac{\Phi^2(p)}{2x}.
\]
Letting $x\rightarrow -\infty$, we have $\lambda(x,p)\rightarrow
+\infty$.
\item When $x>0$, dividing $x$ on both sides of the inequality
\eqref{Sec3_Temp2_Lemma_AE}, we have
\[
\lambda(x,p) \leq -\frac 1 2 x +\frac{\Phi^2(p)}{2x}.
\]
Letting $x\rightarrow +\infty$, we have $\lambda(x,p)\rightarrow
-\infty$.
\end{enumerate}
Obviously $\lambda$ is a continuous function. From the classical
mean value theorem of continuous functions, we know that, for any
$p$, there exists a real number $x(p)$ such that $\lambda(x(p),p)
=0$. We finish the proof of existence.
\end{proof}

\begin{remark}
One conventional approach to prove the solvability of some algebraic
equations is by virtue of a monotone operator with coerciveness (see
for example Zeidler \cite[Theorem 26.A]{Zeidler1990}). Due to
Assumption (H3.3), $\Phi$ is a monotone operator. However, there is
no coercive condition imposed. So this conventional method cannot be
applied for our problem.
\end{remark}

\begin{lemma}\label{Sec3_Lemma_0}
Let Assumptions (H3.1)-(H3.4) hold. For any $(\eta,\phi,\psi,\xi)\in
\mathcal L^{2,K}_{\mathbb F}(0,\infty)$, the forward-backward SDE
\eqref{Sec3_FBSDE_Param_0} admits a unique solution in $\mathcal
L^{2,K}_{\mathbb F}(0,\infty)$.
\end{lemma}

\begin{proof}
Let us consider a linear infinite horizon backward SDE:
\begin{equation}\label{Sec3_FBSDE_Param_0_BSDE}
\begin{aligned}
-dp(t) =\ & \left[ -\mu p(t) +\phi(t) +\eta(t)\right] dt -\left[
(1+\mu)q(t) -\psi(t) \right] dW(t)\\
& -\int_{\mathcal E} \left[ (1+\mu)k(t,e) -\xi(t,e) \right] \tilde
N(dt, de),
\end{aligned}
\end{equation}
and a linear infinite horizon (forward) SDE combined with an
algebraic equation:
\begin{equation}\label{Sec3_FBSDE_Param_0_SDE}
\left\{
\begin{aligned}
dx(t) =\ & \left[ -\mu(x(t)+p(t)) +\phi(t) \right] dt +\left[ -\mu
q(t) +\psi(t) \right] dW(t)\\
& +\int_{\mathcal E} \left[ -\mu k(t,e) +\xi(t,e) \right] \tilde
N(dt, de),\\
x(0) =\ & \Phi(x(0)+p(0)).
\end{aligned}
\right.
\end{equation}
Due to Theorem \ref{Sec2_THM} with $C_1=C_2=0$ and $\rho=\mu$, the
backward SDE \eqref{Sec3_FBSDE_Param_0_BSDE} admits a unique
solution $(p,q,k)\in L^{2,K}_{\mathbb F}(0,\infty;\mathbb R^n)
\times L^{2,K}_{\mathbb F}(0,\infty;\mathbb R^{n\times d}) \times
M^{2,K}_{\mathbb F}(0,\infty;\mathbb R^{n\times l})$. Once $(p,q,k)$
is solved, by Lemma \ref{Sec3_Lemma_AE}, we can uniquely solve
$x(0)$ from the initial condition $x(0) = \Phi(x(0)+p(0))$. then we
solve SDE \eqref{Sec3_FBSDE_Param_0_SDE}. It admits a unique
solution $x$. Next we shall show that $x\in L^{2,K}_{\mathbb
F}(0,\infty;\mathbb R^n)$. For any constant $T>0$, we apply
It\^{o}'s formula to $|x(t)|^2 e^{K t}$ on the finite interval
$[0,T]$:
\[
\begin{aligned}
& \mathbb E\left[ |x(T)|^2 e^{K T} \right] +(2\mu-
K)\mathbb E\int_0^T |x(t)|^2 e^{K t} dt\\
=\ & |x(0)|^2 +\mathbb E\int_0^T \Big[ 2 \langle x(t),\ \phi(t)-\mu
p(t) \rangle +|\psi(t)-\mu q(t)|^2 +\norm{\xi(t,\cdot)-\mu
k(t,\cdot)}^2 \Big] e^{K t} dt.
\end{aligned}
\]
Since $K <2\mu$, then there exists a constant $\varepsilon>0$ such
that $2\mu- K-\varepsilon >0$. By the inequality $2\langle a,\ b
\rangle \leq \varepsilon |a|^2 +(1/\varepsilon)|b|^2$, we have
\[
\begin{aligned}
& \mathbb E\left[ |x(T)|^2 e^{K T} \right] +(2\mu-
K-\varepsilon) \mathbb E\int_0^T |x(t)|^2 e^{ K t} dt\\
\leq\ & |x(0)|^2 +\mathbb E\int_0^T \left[ \frac{1}{\varepsilon}
|\phi(t)-\mu p(t)|^2 +|\psi(t)-\mu q(t)|^2 +\norm{\xi(t,\cdot)-\mu
k(t,\cdot)}^2 \right] e^{ K t} dt.
\end{aligned}
\]
Letting $T\rightarrow\infty$, we have
\[
\begin{aligned}
&(2\mu-
K-\varepsilon) \mathbb E\int_0^\infty |x(t)|^2 e^{K t} dt\\
\leq\ & |x(0)|^2 +\mathbb E\int_0^\infty \left[
\frac{1}{\varepsilon} |\phi(t)-\mu p(t)|^2 +|\psi(t)-\mu q(t)|^2
+\norm{\xi(t,\cdot)-\mu k(t,\cdot)}^2 \right] e^{ K t} dt.
\end{aligned}
\]
We have proved $x\in L^{2, K}_{\mathbb F}(0,\infty;\mathbb R^n)$.

It is easy to verify that $(x^0,y^0,z^0,r^0) = (x,x+p,q,k)$ is a
solution to the forward-backward SDE \eqref{Sec3_FBSDE_Param_0}. We
proved the existence.

We would like to prove the uniqueness in a bigger space: $\mathcal
L^{2}_{\mathbb F}(0,\infty)$. Let $\theta_1(\cdot) =
(x_1(\cdot),y_1(\cdot),z_1(\cdot),r_1(\cdot,\cdot))$ and
$\theta_2(\cdot) =
(x_2(\cdot),y_2(\cdot),z_2(\cdot),r_2(\cdot,\cdot))$ belonging to
$\mathcal L^{2}_{\mathbb F}(0,\infty)$  be two solutions to the
forward-backward SDE \eqref{Sec3_FBSDE_Param_0}. We denote
$\hat\theta(\cdot) = (\hat x(\cdot),\hat y(\cdot),\hat z(\cdot),\hat
r(\cdot,\cdot)) = (x_1(\cdot)-x_2(\cdot), y_1(\cdot)-y_2(\cdot),
z_1(\cdot)-z_2(\cdot), r_1(\cdot,\cdot)-r_2(\cdot,\cdot))$ and apply
It\^{o}'s formula to $\langle \hat x(t),\ \hat y(t) \rangle$ on the
interval $[0,T]$ to get
\[
\mathbb E\Big[ \langle \hat x(T),\ \hat y(T) \rangle \Big] +\mu
\mathbb E\int_0^T |\hat \theta(t)|^2 dt =\Big\langle
\Phi(y_1(0))-\Phi(y_2(0)), \hat y(0) \Big\rangle.
\]
By the monotonicity condition of $\Phi$, we get
\[
\mathbb E\Big[ \langle \hat x(T),\ \hat y(T) \rangle \Big] +\mu
\mathbb E\int_0^T |\hat \theta(t)|^2 dt \leq 0.
\]
Letting $T\rightarrow\infty$, thanks to Proposition \ref{Sec2_Prop}
and Corollary \ref{Sec2_Corollary}, we have
\[
\mu \mathbb E\int_0^\infty |\hat \theta(t)|^2 dt \leq 0.
\]
The uniqueness is proved.
\end{proof}

The above Lemma \ref{Sec3_Lemma_0} shows that, when $\alpha=0$, the
forward-backward SDE \eqref{Sec3_FBSDE_Param} is in a simple form
and then is uniquely solvable. It is clear that, when $\alpha =1$
and $(\eta, \phi, \psi, \xi)$ vanish, the forward-backward SDE
\eqref{Sec3_FBSDE_Param} coincides with \eqref{Sec3_FBSDE}. We will
show that there exists a fixed step-length $\delta_0>0$, such that,
if, for some $\alpha_0\in [0,1)$, \eqref{Sec3_FBSDE_Param} is
uniquely solvable for any $(\eta, \phi, \psi, \xi) \in \mathcal
L^{2,K}_{\mathbb F}(0,\infty)$, then the same conclusion holds for
$\alpha_0$ being replaced by $\alpha_0+\delta \leq 1$ with
$\delta\in [0,\delta_0]$. Once this has been proved, we can increase
the parameter $\alpha$ step by step and finally reach $\alpha =1$,
which gives the unique solvability of the forward-backward SDE
\eqref{Sec3_FBSDE}. This idea is adopted from
\cite{HuPeng1995,PengShi2000,PengWu1999,Yong1997,Yong2010}, and this
method is called the method of continuation.

Now, we prove the following continuation lemma.

\begin{lemma}\label{Sec3_Lemma_Continuation}
Under Assumptions (H3.1)-(H3.4), there exists an absolute constant
$\delta_0>0$ such that, if, for some $\alpha_0\in [0,1)$, the
forward-backward SDE \eqref{Sec3_FBSDE_Param} is uniquely solvable
in $\mathcal L^{2,K}_{\mathbb F}(0,\infty)$ for any $(\eta, \phi,
\psi, \xi) \in \mathcal L^{2,K}_{\mathbb F}(0,\infty)$, then the
same is true for $\alpha = \alpha_0 +\delta$ with $\delta\in
[0,\delta_0]$ and $\alpha_0+\delta\leq 1$.
\end{lemma}

\begin{proof}
Let $\delta_0$ be determined as follows. Let $\delta\in
[0,\delta_0]$. For each $\theta(\cdot) =
(x(\cdot),y(\cdot),z(\cdot),r(\cdot,\cdot)) \in \mathcal
L^{2,K}_{\mathbb F}(0,\infty)$, we consider the following
forward-backward SDE (compared to \eqref{Sec3_FBSDE_Param} with
$\alpha =\alpha_0 +\delta$):
\begin{equation}
\left\{
\begin{aligned}
dX(t) =\ & \Big[ \alpha_0 b(t,\Theta(t)) -\mu(1-\alpha_0) Y(t)
+\delta\big( b(t,\theta(t)) +\mu y(t) \big) +\phi(t) \Big]dt\\
& +\Big[ \alpha_0 \sigma(t,\Theta(t)) -\mu(1-\alpha_0) Z(t)
+\delta\big( \sigma(t,\theta(t)) +\mu z(t) \big) +\psi(t)
\Big]dW(t)\\
& +\int_{\mathcal E} \Big[ \alpha_0 \gamma(t,e,\Theta(t-))
-\mu(1-\alpha_0) R(t,e)\\
& +\delta\big( \gamma(t,e,\theta(t-)) +\mu r(t,e) \big) +\xi(t,e)
\Big] \tilde N(dt,de),\\
-dY(t) =\ & \Big[ \alpha_0 g(t,\Theta(t)) +\mu(1-\alpha_0)X(t)
+\delta\big( g(t,\theta(t))-\mu x(t) \big) +\eta(t) \Big] dt\\
& -Z(t) dW(t) -\int_{\mathcal E} R(t,e) \tilde N(dt,de),\\
X(0) =\ & \Phi(Y(0)).
\end{aligned}
\right.
\end{equation}
It is easy to check that $\big( \delta\big(
g(\cdot,\theta(\cdot))-\mu x(\cdot) \big) +\eta(\cdot),\ \delta\big(
b(\cdot,\theta(\cdot))+\mu y(\cdot) \big) +\phi(\cdot),\ \delta\big(
\sigma(\cdot,\theta(\cdot))+\mu z(\cdot) \big) +\psi(\cdot),\
\delta\big( \gamma(\cdot,\cdot,\theta(\cdot-)) +\mu r(\cdot,\cdot)
\big) +\xi(\cdot,\cdot) \big) \in \mathcal L^{2,K}_{\mathbb
F}(0,\infty)$. Then, by our assumptions, the above forward-backward
SDE is uniquely solvable in the space $\mathcal L^{2,K}_{\mathbb F}
(0,\infty)$. We denote the unique solution by $\Theta(\cdot) =
(X(\cdot),Y(\cdot),Z(\cdot),R(\cdot,\cdot))$. We have established a
mapping
\[
\Theta = I_{\alpha_0+\delta}(\theta) : \mathcal L^{2,K}_{\mathbb
F}(0,\infty) \rightarrow \mathcal L^{2,K}_{\mathbb F}(0,\infty).
\]
Next we shall prove it is a contraction.

Let $\theta_1 = (x_1,y_1,z_1,r_1)$, $\theta_2 = (x_2,y_2,z_2,r_2)
\in \mathcal L^{2,K}_{\mathbb F}(0,\infty)$ and $\Theta_1 =(X_1,
Y_1, Z_1,R_1) = I_{\alpha_0+\delta}(\theta_1)$, $\Theta_2 =(X_2,
Y_2, Z_2,R_2) = I_{\alpha_0+\delta}(\theta_2)$. Let
\[
\begin{aligned}
\hat \theta =\ & (\hat x,\hat y,\hat z,\hat r) = (x_1-x_2, y_1-
y_2, z_1- z_2, r_1- r_2),\\
\hat \Theta =\ & (\hat X,\hat Y,\hat Z,\hat R) = (X_1-X_2, Y_2- Y_2,
Z_1-Z_2, R_1- R_2).
\end{aligned}
\]
For any $T>0$, applying It\^{o}'s formula to $\langle \hat X(t),\
\hat Y(t) \rangle e^{Kt}$ on the interval $[0,T]$, we have
\[
\begin{aligned}
& \mathbb E\Big[ \left\langle \hat X(T),\ \hat Y(T) \right\rangle
e^{KT} \Big] -\left\langle \Phi(Y_1(0))-\Phi(Y_2(0)),\ \hat Y(0)
\right\rangle\\
=\ & \alpha_0 \mathbb E\int_0^T \left\langle
A(t,\Theta_1(t))-A(t,\Theta_2(t)),\ \hat\Theta(t) \right\rangle
e^{Kt} dt -\mu(1-\alpha_0) \mathbb E\int_0^T |\hat \Theta(t)|^2
e^{Kt} dt\\
& +\delta\mathbb E\int_0^T \left\langle
A(t,\theta_1(t))-A(t,\theta_2(t)),\ \hat\Theta(t) \right\rangle
e^{Kt} dt +\delta\mu \mathbb E\int_0^T \left\langle \hat\theta(t),\
\hat\Theta(t) \right\rangle e^{Kt} dt\\
& + K\mathbb E\int_0^T \left\langle \hat X(t),\ \hat Y(t)
\right\rangle e^{Kt} dt.
\end{aligned}
\]
By Assumptions (H3.2) and (H3.3), we deduce
\[
\begin{aligned}
& \mathbb E\Big[ \left\langle \hat X(T),\ \hat Y(T) \right\rangle
e^{KT} \Big] +(\mu -\frac 1 2 K) \mathbb E\int_0^T |\hat\Theta(t)|^2 e^{Kt} dt\\
\leq\ & \delta(C+\mu) \mathbb E\int_0^T |\hat\theta(t)|
|\hat\Theta(t)| e^{Kt} dt.
\end{aligned}
\]
Since $\frac{1}{2}K <\mu$ (see Assumption (H3.4)), then there exists
a constant $\varepsilon>0$ such that $\mu- \frac 1 2 K-\varepsilon
>0$. Then we have
\[
\begin{aligned}
& \mathbb E\Big[ \left\langle \hat X(T),\ \hat Y(T) \right\rangle
e^{KT} \Big] +(\mu -\frac 1 2 K -\varepsilon) \mathbb E\int_0^T |\hat\Theta(t)|^2 e^{Kt} dt\\
\leq\ & \delta^2 \cdot \frac{(C+\mu)^2}{4\varepsilon} \mathbb
E\int_0^T |\hat\theta(t)|^2 e^{Kt} dt.
\end{aligned}
\]
Letting $T\rightarrow\infty$, we have
\[
\mathbb E\int_0^\infty |\hat\Theta(t)|^2 e^{Kt} dt \leq \delta^2
\cdot \frac{(C+\mu)^2}{2\varepsilon(2\mu-K-2\varepsilon)} \mathbb
E\int_0^T |\hat\theta(t)|^2 e^{Kt} dt.
\]
Now we choose $\delta_0^2 =
\frac{\varepsilon(2\mu-K-2\varepsilon)}{2(C+\mu)^2}$, then for any
$\delta\in [0,\delta_0]$, we have the following estimate:
\[
\norm{\hat\Theta(\cdot)}_{\mathcal L^{2,K}_{\mathbb F}} \leq \frac 1
2 \norm{\hat\theta(\cdot)}_{\mathcal L^{2,K}_{\mathbb F}}.
\]
This implies that the mapping $I_{\alpha_0+\delta}$ is a
contraction. Hence, it has a unique fixed point, which is the unique
solution of \eqref{Sec3_FBSDE_Param} for $\alpha =\alpha_0+\delta$.
We complete the proof.
\end{proof}

Now, we give an existence and uniqueness result for the
forward-backward SDE \eqref{Sec3_FBSDE}.

\begin{theorem}\label{Sec3_THM}
Under Assumptions (H3.1)-(H3.4), the forward-backward SDE
\eqref{Sec3_FBSDE} admits a unique solution
$(x(\cdot),y(\cdot),z(\cdot),r(\cdot,\cdot)) \in \mathcal
L^{2,K}_{\mathbb F}(0,\infty)$.
\end{theorem}

\begin{proof}
By Lemma \ref{Sec3_Lemma_0} and Lemma \ref{Sec3_Lemma_Continuation},
we can solve the forward-backward SDE \eqref{Sec3_FBSDE_Param}
uniquely for any $\alpha\in [0,1]$ and $(\eta, \phi, \psi, \xi)\in
\mathcal L^{2,K}_{\mathbb F}(0,\infty)$. Particularly,
\eqref{Sec3_FBSDE_Param} with $\alpha=1$ and $(\eta, \phi, \psi,
\xi) =0$, which is \eqref{Sec3_FBSDE}, admits a unique solution. We
finish the proof.
\end{proof}

\begin{remark}
\begin{enumerate}[(i)]
\item By a similar proof as that of Lemma \ref{Sec3_Lemma_0}, the
uniqueness holds true in the bigger space $\mathcal L^2_{\mathbb
F}(0,\infty)$.
\item By Proposition \ref{Sec2_Prop} and Corollary
\ref{Sec2_Corollary}, the unique solution (x,y,z,r) of the
forward-backward SDE \eqref{Sec3_FBSDE} belongs to $\mathcal
X^K(0,\infty;\mathbb R^n) \times \mathcal X^K(0,\infty;\mathbb R^n)
\times L^{2,K}_{\mathbb F}(0,\infty;\mathbb R^{n\times d}) \times
M^{2,K}_{\mathbb F}(0,\infty;\mathbb R^{n\times l})$ exactly.
\end{enumerate}
\end{remark}

In the rest of this section, we would like to establish some
properties of the solutions to forward-backward SDEs including two
stability results and a comparison theorem. First we establish the
stability results.

\begin{proposition}\label{Sec3.1_Prop_Stability}
Let $(b_1,\sigma_1,\gamma_1,g_1,\Phi_1)$ and
$(b_2,\sigma_2,\gamma_2,g_2,\Phi_2)$ be two sets of coefficients of
forward-backward SDEs satisfying Assumptions (H3.1)-(H3.4). Let
$\theta_1=(x_1,y_1,z_1,r_1)$ and $\theta_2=(x_2,y_2,z_2,r_2)$ be the
corresponding solutions.
\begin{enumerate}[(i)]
\item If we assume that $\Phi_1=\Phi_2$, then
\begin{equation}\label{Sec3.1_L2_1}
\mathbb E\int_0^\infty |\theta_1(t)-\theta_2(t)|^2 e^{Kt} dt \leq C
\mathbb E\int_0^\infty | A_1(t,\theta_2(t)) -A_2(t,\theta_2(t)) |^2
e^{Kt} dt,
\end{equation}
where $C$ is a constant depending on $\mu$ and $K$.
\item If we strengthen the monotonicity condition on $\Phi_1$ as
follows: there exists a constant $\nu>0$ such that for any $y_1$,
$y_2\in\mathbb R^n$,
\begin{equation}\label{Sec3.1_StrongMonot}
\langle \Phi_1(y_1)-\Phi_1(y_2),\ y_1-y_2 \rangle \leq -\nu
|y_1-y_2|^2,
\end{equation}
then
\begin{equation}\label{Sec3.1_L2_2}
\begin{aligned}
& |y_1(0)-y_2(0)|^2 +\mathbb E\int_0^\infty
|\theta_1(t)-\theta_2(t)|^2 e^{Kt} dt\\
& \leq C \left\{ |\Phi_1(y_2(0))-\Phi_2(y_2(0))|^2 + \mathbb
E\int_0^\infty | A_1(t,\theta_2(t)) -A_2(t,\theta_2(t)) |^2 e^{Kt}
dt \right\},
\end{aligned}
\end{equation}
where $C$ is a constant depending on $\mu$, $\nu$ and $K$.
\end{enumerate}
\end{proposition}

\begin{proof}
We apply It\^{o}'s formula to $\langle \hat x(t),\ \hat y(t) \rangle
e^{Kt}$ on the interval $[0,T]$:
\[
\begin{aligned}
& \mathbb E\Big[ \langle\hat x(T),\ \hat y(T)\rangle e^{KT} \Big]
-\langle \Phi_1(y_1(0))-\Phi_1(y_2(0)),\ \hat y(0) \rangle\\
& -\mathbb E\int_0^T \Big\langle A_1(t,\theta_1(t))
-A_1(t,\theta_2(t)),\ \hat \theta(t) \Big\rangle e^{Kt} dt\\
=\ & \langle \Phi_1(y_2(0))-\Phi_2(y_2(0)),\ \hat y(0) \rangle\\
& +\mathbb E\int_0^T \Big\{ \Big\langle
A_1(t,\theta_2(t))-A_2(t,\theta_2(t)), \hat \theta(t) \Big\rangle
+K\Big\langle \hat x(t),\ \hat y(t) \Big\rangle \Big\} e^{Kt} dt,
\end{aligned}
\]
where the notations $\hat x := x_1-x_2$ etc. By the monotonicity
condition on $A_1$, we have
\begin{equation}\label{Sec3.1_Temp}
\begin{aligned}
& \mathbb E\Big[ \langle\hat x(T),\ \hat y(T)\rangle e^{KT} \Big]
-\langle \Phi_1(y_1(0))-\Phi_1(y_2(0)),\ \hat y(0) \rangle
+(\mu-\frac K 2) \mathbb E\int_0^T |\hat\theta(t)|^2 e^{Kt} dt\\
\leq\ & \langle \Phi_1(y_2(0))-\Phi_2(y_2(0)),\ \hat y(0) \rangle
+\mathbb E\int_0^T \Big\langle
A_1(t,\theta_2(t))-A_2(t,\theta_2(t)), \hat \theta(t) \Big\rangle
e^{Kt} dt.
\end{aligned}
\end{equation}

(i) When $\Phi_1=\Phi_2$, considering the monotonicity condition on
$\Phi_1$ (see Assumption (H3.3)), \eqref{Sec3.1_Temp} is reduced to
\[
\begin{aligned}
& \mathbb E\Big[ \langle\hat x(T),\ \hat y(T)\rangle e^{KT} \Big]
+(\mu-\frac K 2) \mathbb E\int_0^T |\hat\theta(t)|^2 e^{Kt} dt\\
\leq\ & \mathbb E\int_0^T \Big\langle
A_1(t,\theta_2(t))-A_2(t,\theta_2(t)), \hat \theta(t) \Big\rangle
e^{Kt} dt.
\end{aligned}
\]
By a similar technique as that in the proof of Lemma
\ref{Sec3_Lemma_Continuation}, we obtain the estimate
\eqref{Sec3.1_L2_1}.

(ii) When $\Phi_1$ satisfies the strong monotonicity condition
\eqref{Sec3.1_StrongMonot}, the inequality \eqref{Sec3.1_Temp} is
reduced to
\[
\begin{aligned}
& \mathbb E\Big[ \langle\hat x(T),\ \hat y(T)\rangle e^{KT} \Big]
+\nu |\hat y(0)|^2
+(\mu-\frac K 2) \mathbb E\int_0^T |\hat\theta(t)|^2 e^{Kt} dt\\
\leq\ & \langle \Phi_1(y_2(0))-\Phi_2(y_2(0)),\ \hat y(0) \rangle
+\mathbb E\int_0^T \Big\langle
A_1(t,\theta_2(t))-A_2(t,\theta_2(t)), \hat \theta(t) \Big\rangle
e^{Kt} dt.
\end{aligned}
\]
Similar to the proof of Lemma \ref{Sec3_Lemma_Continuation}, we get
the estimate \eqref{Sec3.1_L2_2}.
\end{proof}

\begin{remark}
In the proof of the above proposition, it is easy to see, even if
the Lipschitz condition and the monotonicity conditions are not
satisfied by $(b_2,\sigma_2,\gamma_2,g,\Phi)$, the estimates
\eqref{Sec3.1_L2_1} and \eqref{Sec3.1_L2_2} still hold.
\end{remark}

We would like to point out that the $L^2$-estimate (see
\eqref{Sec3.1_L2_1} and \eqref{Sec3.1_L2_2}) for solutions of
forward-backward SDEs will play a key role in studying Pontryagin's
maximum principle and Bellman's dynamic programming principle for
stochastic optimal control and stochastic differential game problems
of infinite horizon forward-backward SDEs.

Next we prove a comparison theorem. As same as before, let $\theta_1
= (x_1,y_1,z_1,r_1)$ and $\theta_2 = (x_2,y_2,z_2,r_2)$ be the
solutions of \eqref{Sec3_FBSDE} with coefficients
$(b,\sigma,\gamma,g,\Phi_1)$ and $(b,\sigma,\gamma,g,\Phi_2)$
respectively. Denote
\[
\hat\theta := \theta_1-\theta_2 = (x_1-x_2,y_1-y_2,z_1-z_2,r_1-z_2)
=: (\hat x,\hat y,\hat z,\hat r).
\]
Similar to Lemma 7 in \cite{PengShi2000}, one can easily prove the
following lemma.

\begin{lemma}\label{Sec3.1_CT_Lemma}
Let Assumptions (H3.1)-(H3.4) holds for $(b,\sigma,\gamma,g,\Phi_1)$
and $(b,\sigma,\gamma,g,\Phi_2)$.
\begin{enumerate}[(i)]
\item For any $t\in [0,\infty)$, we have
\[
\langle \hat x(t),\ \hat y(t) \rangle \geq 0.
\]
\item If we define an $\mathbb F$-stopping time: $\tau = \inf \{ t\geq
0;\ \langle \hat x(t),\ \hat y(t) \rangle =0 \}$, then we further
have
\[
\hat\theta(t)\mathbbm{1}_{[\tau,\infty)}(t) = (\hat x(t),\hat
y(t),\hat z(t),\hat r(t,\cdot))\mathbbm{1}_{[\tau,\infty)}(t) =0.
\]
\end{enumerate}
\end{lemma}

\begin{proof}
(i) Clearly, for any $t\in [0,\infty)$, there exists a sequence of
times $\{T_i\}_{i=1}^\infty$, which increases and diverges as
$i\rightarrow\infty$, such that
\[
\lim_{i\rightarrow\infty} \mathbb E\Big[ \langle \hat x(T_i),\ \hat
y(T_i) \rangle\ \Big|\ \mathcal F_t \Big] =0.
\]
We apply It\^{o}'s formula to $\langle \hat x(s),\ \hat y(s)
\rangle$ on the interval $[t,T_i]$ to have
\[
\begin{aligned}
\mathbb E\Big[ \langle \hat x(T_i),\ \hat y(T_i) \rangle\ \Big|\
\mathcal F_t \Big] -\langle \hat x(t),\ \hat y(t) \rangle  =\ &
\mathbb E\bigg[ \int_t^{T_i} \Big\langle A(s,\theta_1(s))
-A(s,\theta_2(s)),\ \hat\theta(s) \Big\rangle ds\ \bigg|\ \mathcal
F_t \bigg]\\
\leq\ & -\mu \mathbb E \bigg[ \int_t^{T_i} |\hat \theta(s)|^2 ds\
\bigg|\ \mathcal F_t \bigg].
\end{aligned}
\]
Then, letting $i\rightarrow\infty$, we have
\[
\langle \hat x(t),\ \hat y(t) \rangle \geq
\lim_{i\rightarrow\infty}\mu \mathbb E \bigg[ \int_t^{T_i} |\hat
\theta(s)|^2 ds\ \bigg|\ \mathcal F_t \bigg] \geq 0.
\]

(ii) For the given $\mathbb F$-stopping time $\tau$, it is easy to
see, for any $T\in [0,\infty)$,
\[
\langle \hat x(T),\ \hat y(T) \rangle -\langle \hat x(\tau\wedge
T),\ \hat y(\tau\wedge T) \rangle \geq 0.
\]
On the other hand, from It\^{o}'s formula,
\[
\begin{aligned}
\mathbb E\Big[\langle \hat x(T),\ \hat y(T) \rangle -\langle \hat
x(\tau\wedge T),\ \hat y(\tau\wedge T) \rangle\Big] =\ & \mathbb
E\int_{\tau\wedge T}^T \Big\langle
A(s,\theta_1(s))-A(s,\theta_2(s)),\ \hat \theta(s) \Big\rangle ds\\
\leq\ & -\mu  \mathbb E\int_{\tau\wedge T}^T |\hat\theta(s)|^2 ds
\leq 0.
\end{aligned}
\]
Then, by the above two inequalities, we obtain
\[
\hat\theta(s)\mathbbm{1}_{[\tau\wedge T,T]}(t) =0.
\]
Due to the arbitrariness of $T$, we get the desired conclusion
$\hat\theta(s)\mathbbm{1}_{[\tau,\infty)}(t) =0$.
\end{proof}

\begin{theorem}\label{Sec3.1_THM_Comparison}
Let $n=1$. Let Assumptions (H3.1)-(H3.4) hold for
$(b,\sigma,\gamma,g,\Phi_1)$ and $(b,\sigma,\gamma,g,\Phi_2)$.
\begin{enumerate}[(i)]
\item If $\Phi_1(y_2(0))>\Phi_2(y_2(0))$, then $\hat y(0)>0$.
\item If $\Phi_1(y_2(0))=\Phi_2(y_2(0))$, then $\hat y(0)=0$.
\end{enumerate}
\end{theorem}

\begin{proof}
When $n=1$, Lemma \ref{Sec3.1_CT_Lemma}-(i) is read as $\hat
x(t)\hat y(t)\geq 0$ for any $t\in [0,\infty)$. Especially, taking
$t=0$, we have $\hat x(0)\hat y(0) \geq 0$. Then from the
monotonicity condition on $\Phi_1$,
\begin{equation}\label{Sec3.1_CT_Lemma_Temp}
\begin{aligned}
0 \leq\ & \Big( \Phi_1(y_1(0))-\Phi_2(y_2(0)) \Big)\hat y(0)\\
=\ & \Big[ \Big( \Phi_1(y_1(0))-\Phi_1(y_2(0)) \Big) +\Big(
\Phi_1(y_2(0))-\Phi_2(y_2(0)) \Big) \Big] \hat y(0)\\
\leq\ & \Big( \Phi_1(y_2(0))-\Phi_2(y_2(0)) \Big) \hat y(0).
\end{aligned}
\end{equation}

(i) If $\Phi_1(y_2(0))>\Phi_2(y_2(0))$, then
\eqref{Sec3.1_CT_Lemma_Temp} implies $\hat y(0)\geq 0$. Moreover, if
$\hat y(0) =0$, then the stopping time $\tau$ defined in Lemma
\ref{Sec3.1_CT_Lemma}-(ii) is equal to $0$. By Lemma
\ref{Sec3.1_CT_Lemma}-(ii), we have $\hat x(0) =0$. Since $y_1(0) =
y_2(0)$, then
\[
\Phi_1(y_2(0)) -\Phi_2(y_2(0)) = \Phi_1(y_1(0)) -\Phi_2(y_2(0)) =
\hat x(0) =0.
\]
This is a contradiction. Therefore, in this case we must have $\hat
y(0) > 0$.

(ii) When $\Phi_1(y_2(0)) = \Phi_2(y_2(0))$, we also use a framework
of reduction to absurdity to show $\hat y(0)=0$. We assume that
$\hat y(0)\neq 0$. From \eqref{Sec3.1_CT_Lemma_Temp}, we deduce that
\[
\hat x(0) = \Phi_1(y_1(0)) -\Phi_2(y_2(0))= 0.
\]
By Lemma \ref{Sec3.1_CT_Lemma}-(ii) once again, we have $\hat y(0)
=0$. We obtain a contradiction, and then finish the proof.
\end{proof}

\section{Application to backward stochastic LQ problems}\label{Sec4}

In this section, we apply the solvability result of forward-backward
SDEs studied in the above section to deal with two kinds of backward
stochastic linear-quadratic (LQ) problems with jumps, including an
LQ stochastic optimal control (SOC) problem and an LQ nonzero-sum
stochastic differential game (NZSSDG) problem.

Firstly, for the control problem, the system is given by the
following controlled linear backward SDE on the infinite interval
$[0,\infty)$:
\begin{equation}\label{Sec4.1_Sys}
\begin{aligned}
-dy(t) =\ & \left[ A(t)y(t) +\sum_{i=1}^d B_i(t)z_i(t) +\sum_{i=1}^l
\int_{\mathcal E} C_i(t,e) r_i(t,e) \pi_i(de) +D(t)v(t) +\alpha(t)
\right]dt\\
& -\sum_{i=1}^d z_i(t) dW_i(t) -\sum_{i=1}^l \int_{\mathcal E}
r_i(t,e) \tilde N_i(dt,de),
\end{aligned}
\end{equation}
where $A$, $B_i$ ($i=1,2,\dots,d$), $D$ are $\mathbb
F$-progressively measurable, matrix-valued, bounded processes with
appropriate dimensions; $C_i$ ($i=1,2,\dots, l$) is a $\mathcal
P\otimes \mathcal B(\mathcal E)$-measurable, ($n\times n$)
matrix-valued process such that $\int_{\mathcal E} |C_i(t,e)|^2
\pi_i(de)$ is uniformly bounded for any $(\omega,t)\in \Omega\times
[0,\infty)$; and the nonhomogeneous term $\alpha\in L^{2,K}_{\mathbb
F}(0,\infty;\mathbb R^n)$ where $K>0$ is a constant. The admissible
control set is defined by
\[
\begin{aligned}
\mathcal V := \Big\{ & v \in L^{2,K}_{\mathbb F}(0,\infty;\mathbb
R^k)\ \Big|\ \mbox{with respect to $v$, \eqref{Sec4.1_Sys} admits a unique solution}\\
& (y^v,z^v,r^v)\in L^{2,K}_{\mathbb F}(0,\infty;\mathbb R^n) \times
L^{2,K}_{\mathbb F}(0,\infty;\mathbb R^{n\times d}) \times
M^{2,K}_{\mathbb F}(0,\infty;\mathbb R^{n\times l})  \Big\},
\end{aligned}
\]
in which each element $v$ is called an admissible control, and
$(y^v,z^v,r^v)$ is called the state trajectory corresponding to $v$.
In what follows, we will show the admissible control set $\mathcal
V$ is nonempty under some suitable conditions. In addition, we are
given a cost functional associated with $v$ in a quadratic form:
\begin{equation}\label{Sec4.1_Cost}
\begin{aligned}
J(v(\cdot)) =\ & \frac 1 2 \langle Qy(0),\ y(0) \rangle + \frac 1 2
\mathbb E\int_0^\infty \bigg[ \langle L(t)y(t),\ y(t) \rangle
+\sum_{i=1}^d \langle M_i(t)z_i(t),\ z_i(t)
\rangle\\
& +\sum_{i=1}^l \int_{\mathcal E} \langle S_i(t,e)r_i(t,e),\
r_i(t,e) \rangle \pi_i(de) +\langle R(t)v(t),\ v(t) \rangle
\bigg]dt,
\end{aligned}
\end{equation}
where $Q$ is an ($n\times n$) symmetric and positive semi-definite
matrix; $L$, $M_i$ ($i=1,2,\dots,d$) are $\mathbb F$-progressively
measurable, ($n\times n$) symmetric and positive semi-definite
matrix-valued, bounded processes; $S_i$ ($i=1,2,\dots,l$) is a
$\mathcal P\otimes\mathcal B(\mathcal E)$-measurable, ($n\times n$)
symmetric and positive semi-definite matrix-valued, bounded process;
$R$ is an $\mathbb F$-progressively measurable, ($k\times k$)
symmetric and positive definite matrix-valued, bounded process.
Moreover, $R^{-1}$ is also bounded.\\

\noindent{\bf Problem (SOC).} The problem is to find an admissible
control $u\in\mathcal V$ such that
\begin{equation}
J(u(\cdot)) = \inf_{v(\cdot)\in\mathcal V} J(v(\cdot)).
\end{equation}
Such an admissible control $u$ is called an optimal control, and
$(y,z,r) := (y^u,z^u,r^u)$ is called the corresponding optimal state
trajectory.\\

The following result links Problem (SOC) to a forward-backward SDE.

\begin{lemma}\label{Sec4.1_Lemma}
If the following forward-backward SDE:
\begin{equation}\label{Sec4.1_FBSDE_SOC}
\left\{
\begin{aligned}
dx(t) =\ & \Big[ A^\top(t)x(t) -L(t)y(t) \Big] dt +\sum_{i=1}^d
\Big[ B_i^\top(t)x(t) -M_i(t)z_i(t) \Big]dW_i(t)\\
& +\sum_{i=1}^l \int_{\mathcal E} \Big[ C^\top_i(t,e)x(t-) -S_i(t,e)
r_i(t,e) \Big] \tilde N_i(dt,de),\\
-dy(t) =\ & \bigg[ D(t)R^{-1}(t)D^\top(t)x(t) +A(t)y(t)
+\sum_{i=1}^d B_i(t)z_i(t)\\
& +\sum_{i=1}^l\int_{\mathcal E} C_i(t,e)r_i(t,e) \pi_i(de)
+\alpha(t) \bigg] dt -\sum_{i=1}^d z_i(t) dW_i(t)\\
& -\sum_{i=1}^l\int_{\mathcal E} r_i(t,e) \tilde N_i(dt,de),\\
x(0) =\ & -Q y(0)
\end{aligned}
\right.
\end{equation}
admits a solution $(x,y,z,r)\in \mathcal L^{2,K}_{\mathbb
F}(0,\infty)$, then
\begin{equation}\label{Sec4.1_OptimalControl}
u(t) = R^{-1}(t) D^\top(t) x(t),\quad t\in [0,\infty),
\end{equation}
provides an optimal control of Problem (SOC). Moreover the optimal
control is unique.
\end{lemma}

\begin{proof}
First, we prove that $u$ defined by \eqref{Sec4.1_OptimalControl} is
an optimal control for Problem (SOC). For each $v\in\mathcal V$, the
corresponding state trajectory is denoted by $(y^v,z^v,r^v)$. Let us
consider the difference of $J(v(\cdot))$ and $J(u(\cdot))$ (the
argument $(t,e)$ is suppressed):
\begin{equation}\label{Sec4.1_Temp}
\begin{aligned}
J(v(\cdot))-J(u(\cdot)) =\ & \frac 1 2 \Big[ \langle Qy^v(0),\
y^v(0) \rangle -\langle Qy(0),\ y(0) \rangle \Big]\\
& +\frac 1 2 \mathbb E\int_0^\infty \bigg\{ \Big[ \langle Ly^v,\ y^v
\rangle -\langle Ly,\ y \rangle \Big] +\sum_{i=1}^d \Big[ \langle
M_i z^v_i,\ z^v_i \rangle -\langle M_iz_i,\ z_i \rangle \Big]\\
& +\sum_{i=1}^l \int_{\mathcal E} \Big[ \langle S_ir_i^v,\
r_i^v\rangle -\langle S_ir_i,\ r_i \rangle \Big] \pi_i(de) +\Big[
\langle Rv,\ v \rangle -\langle Ru,\ u \rangle \Big] \bigg\} dt\\
=\ & \frac 1 2 \langle Q(y^v(0)-y(0)),\ y^v(0)-y(0) \rangle\\
& +\frac 1 2 \mathbb E\int_0^\infty \bigg\{ \langle L(y^v-y),\ y^v-y
\rangle +\sum_{i=1}^d \langle M_i(z^v_i-z_i),\ z^v_i-z_i \rangle\\
& +\sum_{i=1}^l \int_{\mathcal E} \langle S_i(r_i^v-r_i),\ r_i^v-r_i
\rangle \pi_i(de) +\langle R(v-u),\ v-u \rangle \bigg\} dt +\Lambda,
\end{aligned}
\end{equation}
where
\[
\begin{aligned}
\Lambda =\ & \langle Qy(0),\ y^v(0)-y(0) \rangle +\mathbb
E\int_0^\infty \bigg\{ \langle Ly,\ y^v-y \rangle +\sum_{i=1}^d
\langle M_iz_i,\ z_i^v-z_i \rangle\\
& +\sum_{i=1}^l \int_{\mathcal E} \langle S_ir_i,\ r_i^v-r_i \rangle
\pi_i(de) +\langle Ru,\ v-u \rangle \bigg\} dt.
\end{aligned}
\]
Applying It\^{o}'s formula to $\langle x(t),\ y^v(t)-y(t) \rangle$
on the interval $[0,T]$, by the initial condition of $x$ (see
\eqref{Sec4.1_FBSDE_SOC}) and the definition of $u$ (see
\eqref{Sec4.1_OptimalControl}), we have
\[
\begin{aligned}
& \mathbb E\Big[\langle x(T),\ y^v(T)-y(T) \rangle\Big] +\langle
Qy(0),\ y^v(0)-y(0) \rangle\\
=\ & -\mathbb E\int_0^T \bigg\{ \langle Ly,\ y^v-y \rangle
+\sum_{i=1}^d \langle M_iz_i,\ z_i^v-z_i \rangle +\sum_{i=1}^l
\int_{\mathcal E}\langle S_ir_i,\ r_i^v-r_i \rangle\pi_i(de)
+\langle Ru,\ v-u \rangle \bigg\} dt.
\end{aligned}
\]
Letting $T\rightarrow\infty$, we get $\Lambda =0$. Then, since $Q$,
$L$, $M_i$ ($i=1,2,\dots,d$), $S_i$ ($i=1,2,\dots,l$) are positive
semi-definite, and $R$ is positive definite, we have
$J(v(\cdot))-J(u(\cdot))\geq 0$. Due to the arbitrariness of $v$, we
prove that $u$ defined by \eqref{Sec4.1_OptimalControl} is an
optimal control.

For the uniqueness, besides $u$ given by
\eqref{Sec4.1_OptimalControl}, let $\bar u\in\mathcal V$ be another
optimal control, and denote by $(y^{\bar u},z^{\bar u},r^{\bar u})$
the corresponding optimal state trajectory. Obviously $J(\bar
u(\cdot)) = J(u(\cdot))$. Coming back to \eqref{Sec4.1_Temp}, we
have
\[
\begin{aligned}
0 =\ & \frac 1 2 \langle Q(y^{\bar u}(0)-y(0)),\ y^{\bar u}(0)-y(0) \rangle\\
& +\frac 1 2 \mathbb E\int_0^\infty \bigg\{ \langle L(y^{\bar
u}-y),\ y^{\bar u}-y
\rangle +\sum_{i=1}^d \langle M_i(z^{\bar u}_i-z_i),\ z^{\bar u}_i-z_i \rangle\\
& +\sum_{i=1}^l \int_{\mathcal E} \langle S_i(r_i^{\bar u}-r_i),\
r_i^{\bar u}-r_i \rangle \pi_i(de) +\langle R(\bar u-u),\ \bar u-u
\rangle \bigg\} dt \\
\geq\ & \frac 1 2 \mathbb E\int_0^\infty \langle R(\bar u-u),\ \bar
u-u \rangle dt.
\end{aligned}
\]
Because $R$ is positive definite, we get $\bar u(\cdot) = u(\cdot)$.
We have proved the uniqueness of the optimal control.
\end{proof}

In order to obtain the solvability of \eqref{Sec4.1_FBSDE_SOC}, we
assume the following assumptions:
\begin{enumerate}[(\mbox{A1}.1)]
\item There exists a constant $\mu>0$ such that for any $(\omega,t,e)\in \Omega\times
[0,\infty)\times\mathcal E$, any $i=1,2,\dots,d$, any
$j=1,2,\dots,l$,
\[
D(t)R^{-1}(t)D^\top(t) \geq \mu I,\quad L(t) \geq \mu I,\quad
M_i(t)\geq \mu I, \quad S_j(t,e)\geq \mu I.
\]
\item $2\mu-K\geq 0$.
\end{enumerate}

\noindent Here, $I$ denotes the ($n\times n$) identity matrix and
the expression $A\geq B$ means $A-B$ is positive semi-definite as
usual. We notice that, in the viewpoint of Remark
\ref{Sec3_Remark_Artificial}, Assumption (A1.2) is artificial. If it
does not hold true, then we can consider Problem (SOC) in some
larger space. However, for the convenience of presentation, we keep
it here.

\begin{theorem}\label{Sec4.1_THM}
Under Assumptions (A1.1)-(A1.2), the forward-backward SDE
\eqref{Sec4.1_FBSDE_SOC} admits a unique solution $(x,y,z,r) \in
\mathcal L^{2,K}_{\mathbb F}(0,\infty)$. Moreover, $u$ defined by
\eqref{Sec4.1_OptimalControl} is the unique optimal control for
Problem (SOC).
\end{theorem}

\begin{proof}
It is easy to check that Assumptions (A1.1)-(A1.2) imply Assumptions
(H3.1)-(H3.4). Then by Theorem \ref{Sec3_THM}, the forward-backward
SDE \eqref{Sec4.1_FBSDE_SOC} is uniquely solvable. Moreover, thanks
to Lemma \ref{Sec4.1_Lemma}, we can finish the proof.
\end{proof}

Next we extend the LQ SOC problem to an LQ nonzero-sum stochastic
differential game (NZSSDG) problem. Without loss of generality, we
only consider the case of two players in this paper. The case of
$n(\geq 3)$ players can be treated in the same way. In detail, the
game system is described by the following controlled linear backward
SDE on $[0,\infty)$:
\begin{equation}\label{Sec4.2_Sys}
\begin{aligned}
-dy(t) =\ & \bigg[ A(t)y(t) +\sum_{i=1}^d B_i(t)z_i(t) +\sum_{i=1}^l
\int_{\mathcal E} C_i(t,e) r_i(t,e) \pi_i(de) +D_1(t)v_1(t)\\
& +D_2(t)v_2(t) +\alpha(t) \bigg]dt -\sum_{i=1}^d z_i(t) dW_i(t)
-\sum_{i=1}^l \int_{\mathcal E} r_i(t,e) \tilde N_i(dt,de),
\end{aligned}
\end{equation}
where $A$, $B_i$ ($i=1,2,\dots,d$), $D_1$, $D_2$ are $\mathbb
F$-progressively measurable, matrix-valued, bounded processes with
appropriate dimensions; $C_i$ ($i=1,2,\dots, l$) is a $\mathcal
P\otimes \mathcal B(\mathcal E)$-measurable, ($n\times n$)
matrix-valued process such that $\int_{\mathcal E} |C_i(t,e)|^2
\pi_i(de)$ is uniformly bounded for any $(\omega,t)\in \Omega\times
[0,\infty)$; and the nonhomogeneous term $\alpha\in L^{2,K}_{\mathbb
F}(0,\infty;\mathbb R^n)$ where $K>0$ is a constant. $v_1$ and $v_2$
are the control processes of Player 1 and Player 2, respectively. We
introduce the admissible control set for the two players:
\[
\begin{aligned}
\mathcal V := \Big\{ & (v_1,v_2) \in L^{2,K}_{\mathbb
F}(0,\infty;\mathbb
R^{k_1}) \times L^{2,K}_{\mathbb F}(0,\infty;\mathbb R^{k_2})\ \Big|\ \mbox{with respect to $(v_1,v_2)$, \eqref{Sec4.2_Sys} admits}\\
& \mbox{a unique solution } (y^{v_1,v_2},z^{v_1,v_2},r^{v_1,v_2})\in
L^{2,K}_{\mathbb F}(0,\infty;\mathbb R^n) \times L^{2,K}_{\mathbb
F}(0,\infty;\mathbb R^{n\times d})\\
& \times M^{2,K}_{\mathbb F}(0,\infty;\mathbb R^{n\times l})
\Big\},
\end{aligned}
\]
in which each element $(v_1,v_2)$ is called an admissible control
pair, and $(y^{v_1,v_2},z^{v_1,v_2},r^{v_1,v_2})$ is called the
state trajectory corresponding to $(v_1,v_2)$. As same as Problem
(SOC), we will show the admissible control set $\mathcal V$ is
nonempty under suitable conditions. For any $(v_1,v_2)\in \mathcal
V$, let
\[
\begin{aligned}
\mathcal V_1(v_2) =\ & \{ \bar v_1 \ |\ (\bar v_1, v_2) \in \mathcal
V \},\\
\mathcal V_2(v_1) =\ & \{ \bar v_2 \ |\ (v_1, \bar v_2) \in \mathcal
V \}.
\end{aligned}
\]
Additionally, the cost functionals of the two players are given as
follows: for $i=1,2$,
\begin{equation}\label{Sec4.2_Cost}
\begin{aligned}
J_i(v_1(\cdot),v_2(\cdot)) =\ & \frac 1 2 \langle Q_iy(0),\ y(0)
\rangle + \frac 1 2 \mathbb E\int_0^\infty \bigg[ \langle
L_i(t)y(t),\ y(t) \rangle +\sum_{j=1}^d \langle M_{ij}(t)z_j(t),\
z_j(t) \rangle\\
& +\sum_{j=1}^l \int_{\mathcal E} \langle S_{ij}(t,e)r_j(t,e),\
r_j(t,e) \rangle \pi_j(de) +\langle R_i(t)v_i(t),\ v_i(t) \rangle
\bigg]dt,
\end{aligned}
\end{equation}
where $Q_i$ is an ($n\times n$) symmetric and positive semi-definite
matrix; $L_i$, $M_{ij}$ ($j=1,2,\dots,d$) are $\mathbb
F$-progressively measurable, ($n\times n$) symmetric and positive
semi-definite matrix-valued, bounded processes; $S_{ij}$
($j=1,2,\dots,l$) is a $\mathcal P\otimes\mathcal B(\mathcal
E)$-measurable, ($n\times n$) symmetric and positive semi-definite
matrix-valued, bounded process; $R_i$ is an $\mathbb
F$-progressively measurable, ($k\times k$) symmetric and positive
definite matrix-valued, bounded process. Moreover, $R_i^{-1}$ is
also bounded.

Suppose each player hopes to minimize his/her cost functional
$J_i(v_1(\cdot),v_2(\cdot))$ by selecting an appropriate admissilbe
control $v_i$ ($i=1,2$), then the game problem is formulated
as follows.\\

\noindent{\bf Problem (NZSSDG).} The problem is to find a pair of
admissible controls $(u_1,u_2)\in\mathcal V$ such that
\begin{equation}\label{Sec4.2_Nash}
\begin{aligned}
& J_1(u_1(\cdot),u_2(\cdot)) = \inf_{v_1(\cdot)\in\mathcal V_1(u_2)}
J_1(v_1(\cdot),u_2(\cdot)),\\
& J_2(u_1(\cdot),u_2(\cdot)) = \inf_{v_2(\cdot)\in\mathcal V_2(u_1)}
J_2(u_1(\cdot),v_2(\cdot)).
\end{aligned}
\end{equation}
Such a pair of admissible controls $(u_1,u_2)$ is called a Nash
equilibrium point. For the sake of notations, we denote the state
trajectory corresponding to $(u_1,u_2)$ by $(y,z,r) :=
(y^{u_1,u_2},z^{u_1,u_2},r^{u_1,u_2})$.

\begin{lemma}\label{Sec4.2_Lemma}
If the following forward-backward SDE:
\begin{equation}\label{Sec4.2_FBSDE_NZSSDG}
\left\{
\begin{aligned}
dx_1(t) =\ & \Big[ A^\top(t)x_1(t) -L_1(t)y(t) \Big]dt +\sum_{j=1}^d
\Big[ B_j^\top(t)x_1(t) -M_{1j}(t)z_j(t) \Big] dW_j(t)\\
& +\sum_{j=1}^l \int_{\mathcal E} \Big[ C_j^\top(t,e)x_1(t-)
-S_{1j}(t,e) r_j(t,e) \Big] \tilde N_j(dt,de),\\
dx_2(t) =\ & \Big[ A^\top(t)x_2(t) -L_2(t)y(t) \Big]dt +\sum_{j=1}^d
\Big[ B_j^\top(t)x_2(t) -M_{2j}(t)z_j(t) \Big] dW_j(t)\\
& +\sum_{j=1}^l \int_{\mathcal E} \Big[ C_j^\top(t,e)x_2(t-)
-S_{2j}(t,e) r_j(t,e) \Big] \tilde N_j(dt,de),\\
-dy(t) =\ & \bigg[ D_1(t)R_1^{-1}(t)D_1^\top(t)x_1(t)
+D_2(t)R_2^{-1}(t)D_2^\top(t)x_2(t) +A(t)y(t)\\
& +\sum_{j=1}^d B_j(t)z_j(t) +\sum_{j=1}^l \int_{\mathcal E}
C_j(t,e)r_j(t,e)\pi_j(de) +\alpha(t) \bigg] dt\\
& -\sum_{j=1}^d z_j(t)dW_j(t) -\sum_{j=1}^l \int_{\mathcal E}
r_j(t,e) \tilde N_j(dt,de),\\
x_1(0) =\ & -Q_1 y(0), \quad x_2(0) = -Q_2 y(0)
\end{aligned}
\right.
\end{equation}
admits a solution $(x_1,x_2,y,z,r)\in L^{2,K}_{\mathbb
F}(0,\infty;\mathbb R^n)\times \mathcal L^{2,K}_{\mathbb
F}(0,\infty)$, then
\begin{equation}\label{Sec4.2_NashPoint}
\left(\begin{array}{ccc} u_1(t) \\ u_2(t) \end{array}\right) =
\left(\begin{array}{ccc} R_1^{-1}(t)D_1^\top(t) x_1(t)\\
R_2^{-1}(t)D_2^\top(t) x_2(t)
\end{array}\right),\quad t\in [0,\infty),
\end{equation}
provides a Nash equilibrium point for Problem (NZSSDG).
\end{lemma}

\begin{proof}
We shall link Problem (NZSSDG) with two LQ SOC problems. Precisely,
for $i=1,2$, we fix $u_{3-i}(\cdot)$ which is defined in
\eqref{Sec4.2_NashPoint}. To minimize (the argument $(t,e)$ is
suppressed)
\begin{equation}\label{Sec4.2_Cost_i}
\begin{aligned}
J_i(v_i(\cdot),u_{3-i}(\cdot)) =\ & \frac{1}{2} \langle Q_i
y^{v_i}(0),\ y^{v_i}(0) \rangle +\frac 1 2 \mathbb E\int_0^\infty
\bigg[ \langle L_iy^{v_i},\
y^{v_i} \rangle +\sum_{j=1}^d \langle M_{ij}z_j^{v_i},\ z_j^{v_i} \rangle\\
& +\sum_{j=1}^l\int_{\mathcal E} \langle S_{ij} r_j^{v_i},\
r_j^{v_i} \rangle \pi_j(de) +\langle R_iv_i,\ v_i \rangle \bigg]dt
\end{aligned}
\end{equation}
subject to
\begin{equation}\label{Sec4.2_Sys_i}
\begin{aligned}
-dy^{v_i} =\ & \bigg[ Ay^{v_i} +\sum_{j=1}^d B_jz_j^{v_i}
+\sum_{j=1}^l\int_{\mathcal E} C_jr_j^{v_i} \pi_j(de) +D_iv_i +\Big(
D_{3-i}u_{3-i} +\alpha \Big)
\bigg]dt\\
& -\sum_{j=1}^d z_j^{v_i} dW_j -\sum_{j=1}^l \int_{\mathcal E}
r_j^{v_i} \tilde N_j(dt,de)
\end{aligned}
\end{equation}
over $\mathcal V_i(u_{3-i})$ is an LQ SOC problem. Since
\eqref{Sec4.2_FBSDE_NZSSDG} admits a solution $(x_1,x_2,y,z,r)$,
then $(x_i,y,z,r)$ solves the following forward-backward SDE:
\begin{equation}\label{Sec4.2_FBSDE_SOC_i}
\left\{
\begin{aligned}
dx_i =\ & \Big[ A^\top x_i - L_iy \Big] dt +\sum_{j=1}^d \Big[
B_j^\top x_i -M_{ij}z_j \Big]dW_j\\
& +\sum_{j=1}^l \int_{\mathcal E}
\Big[ C_j^\top x_i -S_{ij}r_j \Big] \tilde N_j(dt,de),\\
-dy =\ & \bigg[ D_iR_i^{-1}D_i^\top x_i +Ay +\sum_{j=1}^d B_jz_j
+\sum_{j=1}^l\int_{\mathcal E} c_jr_j\pi_j(de) +\Big( D_{3-i}u_{3-i}
+\alpha \Big) \bigg]dt\\
& -\sum_{j=1}^d z_j dW_j -\sum_{j=1}^l\int_{\mathcal E} r_j\tilde
N_j(dt,de),\\
x_i(0) =\ & -Q_i y(0).
\end{aligned}
\right.
\end{equation}
By Lemma \ref{Sec4.1_Lemma}, the LQ SOC problem
\eqref{Sec4.2_Cost_i}-\eqref{Sec4.2_Sys_i} admits a unique optimal
control with the form
\[
u_i(t) = R_i^{-1}(t)D_i^\top(t)x_i(t),\quad t\in [0,\infty),
\]
which is coincided with \eqref{Sec4.2_NashPoint}. In other words,
the following equation holds:
\[
J_i(u_i(\cdot),u_{3-i}(\cdot)) = \inf_{v_i(\cdot)\in\mathcal
V_i(u_{3-i})} J_i(v_i(\cdot),u_{3-i}(\cdot)).
\]
Since $i=1,2$, from the definition of the Nash equilibrium point
(see \eqref{Sec4.2_Nash}), $(u_1,u_2)$ defined by
\eqref{Sec4.2_NashPoint} provides a Nash equilibrium point for
Problem (NZSSDG).
\end{proof}

The forward-backward SDE \eqref{Sec4.2_FBSDE_NZSSDG} is more
complicated. In order to obtain the solvability of
\eqref{Sec4.2_FBSDE_NZSSDG}, we would like to employ a linear
transform which is originally introduced by Hamad\`{e}ne
\cite{Hamadene1998} (see also Yu \cite{Yu2012}). For this transform,
we need to introduce the following assumptions:
\begin{enumerate}[(\mbox{A2}.1)]
\item The matrix-valued processes $D_iR_i^{-1}D_i^\top$, $i=1,2$,
are independent of $t$.
\item The following commutation relations among matrices hold true:
\[
D_i(t)R_i^{-1}(t)D_i^\top(t) H(t) =H(t)
D_i(t)R_i^{-1}(t)D_i^\top(t),\quad t\in [0,\infty),\quad i=1,2,
\]
where $H(t) = A^\top(t)$, $B^\top_j(t)$ ($j=1,2,\dots,d$),
$C^\top_j(t)$ ($j=1,2,\dots,l$).
\item There exists a constant $\delta>0$ such that, for any
$(\omega,t)\in \Omega\times [0,\infty)$,
\[
2A(t) +\sum_{j=1}^d B_j(t)B_j^\top(t) +\sum_{j=1}^l\int_{\mathcal E}
C_j(t,e)C_j^\top(t,e) \pi_j(de) +KI \leq -\delta I.
\]
\end{enumerate}
We notice that, Assumption (A2.3) is not necessary when the
corresponding finite horizon game problems were studied (see for
example \cite{Hamadene1998,Yu2012}). Here we assume it due to the
infinite time horizon.

Now we introduce another forward-backward SDE (the argument $(t,e)$
is suppressed):
\begin{equation}\label{Sec4.2_FBSDE_Hama}
\left\{
\begin{aligned}
d\bar x =\ & \Big[ A^\top \bar x -\left( D_1R_1^{-1}D_1^\top L_1 +
D_2R_2^{-1}D_2^\top L_2 \right)\bar y \Big]dt\\
& +\sum_{j=1}^d \Big[ B_j^\top \bar x -\left( D_1R_1^{-1}D_1^\top
M_{1j} + D_2R_2^{-1}D_2^\top M_{2j} \right)\bar z_j \Big] dW_j\\
& +\sum_{j=1}^l \int_{\mathcal E} \Big[ C_j^\top \bar x -\left(
D_1R_1^{-1}D_1^\top S_{1j} + D_2R_2^{-1}D_2^\top S_{2j} \right)\bar
r_j \Big] \tilde N_j(dt,de),\\
-d\bar y =\ & \bigg[ \bar x +A\bar y +\sum_{j=1}^d B_j\bar z_j
+\sum_{j=1}^l \int_{\mathcal E} C_j\bar r_j\pi_j(de) +\alpha \bigg]
dt\\
& -\sum_{j=1}^d \bar z_jdW_j -\sum_{j=1}^l \bar r_j \tilde
N_j(dt,de),\\
\bar x(0) =\ & -\left( D_1(0)R_1^{-1}(0)D_1^\top(0)Q_1 +
D_2(0)R_2^{-1}(0)D_2^\top(0)Q_2 \right) \bar y(0),
\end{aligned}
\right.
\end{equation}
and give the following result.

\begin{lemma}\label{Sec4.2_Lemma_Hama}
Under Assumptions (A2.1)-(A2.3), the existence and uniqueness of
\eqref{Sec4.2_FBSDE_NZSSDG} are equivalent to that of
\eqref{Sec4.2_FBSDE_Hama}.
\end{lemma}

\begin{proof}
On the one hand, by Assumptions (A2.1) and (A2.2), if
$(x_1,x_2,y,z,r)\in L^{2,K}_{\mathbb F}(0,\infty;\mathbb R^n)\times
\mathcal L^{2,K}_{\mathbb F}(0,\infty)$ is a solution of
\eqref{Sec4.2_FBSDE_NZSSDG}, then
\[
\left\{
\begin{aligned}
\bar x(t) =\ & D_1(t)R_1^{-1}(t)D_1^\top(t)x_1(t) +
D_2(t)R_2^{-1}(t)D_2^\top(t)x_2(t),\\
\bar y(t) =\ & y(t),\quad \bar z_j(t) =z_j(t)\ (j=1,2,\dots,d),\quad
\bar r_j(t,e) =r_j(t,e)\ (j=1,2,\dots,l),
\end{aligned}
\right.\quad t\in [0,\infty)
\]
belonging to $\mathcal L^{2,K}_{\mathbb F}(0,\infty)$ solves
\eqref{Sec4.2_FBSDE_Hama}.

On the other hand, if $(\bar x, \bar y, \bar z, \bar r)\in\mathcal
L^{2,K}_{\mathcal F}(0,\infty)$ is a solution of
\eqref{Sec4.2_FBSDE_Hama}, we let $y=\bar y$, $z=\bar z$, $r=\bar r$
and $(x_1, x_2)$ be the unique solution of the following SDE:
\[
\left\{
\begin{aligned}
dx_1 =\ & \Big[ A^\top x_1 -L_1y \Big]dt +\sum_{j=1}^d \Big[
B_j^\top x_1 -M_{1j}z_j \Big] dW_j +\sum_{j=1}^l \int_{\mathcal E}
\Big[ C_j^\top x_1 -S_{1j}r_j \Big] \tilde N_j(dt,de),\\
dx_2 =\ & \Big[ A^\top x_2 -L_2y \Big]dt +\sum_{j=1}^d \Big[
B_j^\top x_2 -M_{2j}z_j \Big] dW_j +\sum_{j=1}^l \int_{\mathcal E}
\Big[ C_j^\top x_2 -S_{2j}r_j \Big] \tilde N_j(dt,de),\\
x_1(0) =\ & -Q_1y(0),\quad x_1(0) = -Q_2 y(0).
\end{aligned}
\right.
\]
Obviously $(x_1,x_2)\in S^{2,loc}_{\mathbb F}(0,\infty;\mathbb
R^n)\times S^{2,loc}_{\mathbb F}(0,\infty;\mathbb R^n)$. Moreover,
Assumption (A2.3) ensures it also belongs to $L^{2,K}_{\mathbb
F}(0,\infty;\mathbb R^n) \times L^{2,K}_{\mathbb F}(0,\infty;\mathbb
R^n)$. In fact, for $i=1,2$, we apply It\^{o}'s formula to
$|x_i(t)|^2 e^{Kt}$ on the interval $[0,T]$ to have
\[
\begin{aligned}
& \mathbb E\Big[ |x_i(T)|^2 e^{KT} \Big] -|x_i(0)|^2\\
%=\ & \mathbb E\int_0^T \bigg[ 2\langle Ax_i,\ x_i \rangle -2\langle
%x_i,\ L_iy \rangle +\sum_{j=1}^d |B_j^\top x_i -M_{ij}z_j|^2\\
%& +\sum_{j=1}^l \int_{\mathcal E} |C_j^\top x_i -S_{ij}r_j|^2
%\pi_j(de) +K |x_i|^2 \bigg] e^{Kt} dt\\
=\ & \mathbb E\int_0^T \bigg[ \Big\langle \Big( 2A +\sum_{j=1}^d
B_jB_j^\top +\sum_{j=1}^l\int_{\mathcal E} C_jC_j^\top\pi_j(de) +KI
\Big)x_i,\ x_i \Big\rangle\\
& +2\Big\langle \Big( -L_iy +\sum_{j=1}^d B_jM_{ij}z_j
+\sum_{j=1}^l\int_{\mathcal E} C_jS_{ij}r_j \pi_j(de) \Big),\ x_i
\Big\rangle\\
& +\sum_{j=1}^d |M_{ij}z_j|^2 +\sum_{j=1}^l\int_{\mathcal E}
|S_{ij}r_j|^2 \pi_j(de) \bigg] e^{Kt} dt.
\end{aligned}
\]
By Assumption (A2.3) and the inequality: $2\langle y,\ x \rangle
\leq (2/\delta)|y|^2 +(\delta/2)|x|^2$,
\[
\begin{aligned}
& \frac{\delta}{2} \mathbb E\int_0^T |x_i|^2 e^{Kt} dt\\
\leq\ & |x_i(0)|^2 +\mathbb E\int_0^T \bigg[ \frac{2}{\delta}\bigg|
-L_iy +\sum_{j=1}^d B_jM_{ij}z_j +\sum_{j=1}^l\int_{\mathcal E}
C_jS_{ij}r_j \pi_j(de) \bigg|^2\\
& +\sum_{j=1}^d |M_{ij}z_j|^2 +\sum_{j=1}^l\int_{\mathcal E}
|S_{ij}r_j|^2 \pi_j(de) \bigg] e^{Kt} dt.
\end{aligned}
\]
Letting $T\rightarrow\infty$, we get $x_i\in L^{2,K}_{\mathbb
F}(0,\infty;\mathbb R^n)$. In what follows, we shall show that
$(x_1,x_2,y,z,r)\in L^{2,K}_{\mathbb F}(0,\infty;\mathbb R^n)\times
\mathcal L^{2,K}_{\mathbb F}(0,\infty)$ defined above is a solution
of the forward-backward SDE \eqref{Sec4.2_FBSDE_NZSSDG}. Actually,
the remaining thing is to show $(x_1,x_2,y,z,r)$ satisfies the
backward equation in \eqref{Sec4.2_FBSDE_NZSSDG}. Compared with the
backward equation in \eqref{Sec4.2_FBSDE_Hama}, we only need to show
$D_1R_1^{-1} D_1^\top x_1 +D_2R_2^{-1} D_2^\top x_2 = \bar x$. For
the convenience, we let
\[
\tilde x(t) = D_1(t)R_1^{-1}(t) D_1^\top(t) x_1(t)
+D_2(t)R_2^{-1}(t) D_2^\top(t) x_2(t).
\]
By Assumptions (A2.1) and (A2.2), we know $\tilde x$ satisfies
\[
\left\{
\begin{aligned}
d\tilde x =\ & \Big[ A^\top \tilde x -\left( D_1R_1^{-1}D_1^\top L_1
+ D_2R_2^{-1}D_2^\top L_2 \right) y \Big]dt\\
& +\sum_{j=1}^d \Big[ B_j^\top \tilde x -\left( D_1R_1^{-1}D_1^\top
M_{1j} + D_2R_2^{-1}D_2^\top M_{2j} \right) z_j \Big] dW_j\\
& +\sum_{j=1}^l \int_{\mathcal E} \Big[ C_j^\top \tilde x -\left(
D_1R_1^{-1}D_1^\top S_{1j} + D_2R_2^{-1}D_2^\top S_{2j} \right)
r_j \Big] \tilde N_j(dt,de),\\
\tilde x(0) =\ & -\left( D_1(0)R_1^{-1}(0)D_1^\top(0)Q_1 +
D_2(0)R_2^{-1}(0)D_2^\top(0)Q_2 \right) y(0),
\end{aligned}
\right.
\]
which coincides with the forward equation in
\eqref{Sec4.2_FBSDE_Hama}. Regarding $(y,z,r)$ as fixed processes,
from the uniqueness of SDE, we have $\bar x =\tilde x$ and we proved
that $(x_1,x_2,y,z,r)$ solves \eqref{Sec4.2_FBSDE_NZSSDG}.

In a similar way, one can prove that the uniqueness of
\eqref{Sec4.2_FBSDE_NZSSDG} is equivalent to that of
\eqref{Sec4.2_FBSDE_Hama}.
\end{proof}

For the solvability of \eqref{Sec4.2_FBSDE_Hama}, and then
\eqref{Sec4.2_FBSDE_NZSSDG}, we impose the following assumptions:
\begin{enumerate}[(\mbox{A3}.1)]
\item The matrix $D_1(0)R_1^{-1}(0)D_1^\top(0) Q_1 +D_2(0)R_2^{-1}(0)D_2^\top (0)
Q_2$ is positive semi-definite.
\item There exists a constant $\mu>0$ such that
\[
D_1(t)R_1^{-1}(t)D_1^\top(t) H_1(t) + D_2(t)R_2^{-1}(t)D_2^\top(t)
H_2(t) \geq \mu I, \quad t\in [0,\infty),
\]
where $H_i(t) = L_i$, $M_{ij}$ ($j=1,2,\dots,d$), $S_{ij}$
($j=1,2,\dots,l$), $i=1,2$.
\item $2\min\{\mu, 1\}-K \geq 0$.
\end{enumerate}
Once again, in the viewpoint of Remark \ref{Sec3_Remark_Artificial},
Assumption (A3.3) is artificial. We keep it here for the convenience
of presentation.

\begin{remark}
\begin{enumerate}[(i)]
\item For the symmetric matrices $D_i(0)R_i^{-1}(0)D_i^\top(0)$ and
$Q_i$ ($i=1,2$), if they are commutative, then there exists an
orthogonal matrix $P$ such that both
$P^{-1}D_i(0)R_i^{-1}(0)D_i^\top(0)P$ and $P^{-1}Q_iP$ are diagonal
matrices. Moreover, due to the semi-definiteness of
$D_i(0)R_i^{-1}(0)D_i^\top(0)$ and $Q_i$, we have the following
statement: if $D_i(0)R_i^{-1}(0)D_i^\top(0) Q_i = Q_i
D_i(0)R_i^{-1}(0)D_i^\top(0)$ ($i=1,2$), then (A3.1) holds true.
\item Similarly, if there exist two constants $\beta_1\geq 0$ and $\beta_2\geq
0$ satisfying $\beta_1+\beta_2>0$ such that
\[
\begin{aligned}
& H_i(t)\geq \beta_i I,\quad D_i(t)R_i^{-1}(t)D_i^\top(t) \geq
\beta_i I,\\
& D_i(t)R_i^{-1}(t)D_i^\top(t) H_i(t) = H_i(t)
D_i(t)R_i^{-1}(t)D_i^\top(t),
\end{aligned}
\]
where $H_i(t) = L_i$, $M_{ij}$ ($j=1,2,\dots,d$), $S_{ij}$
($j=1,2,\dots,l$), $i=1,2$, then (A3.2) holds true.
\end{enumerate}
\end{remark}

\begin{theorem}\label{Sec4.2_THM}
Let Assumptions (A2.1)-(A2.3) and (A3.1)-(A3.3) hold.
\begin{enumerate}[(i)]
\item The forward-backward SDE
\eqref{Sec4.2_FBSDE_NZSSDG} admits a unique solution
$(x_1,x_2,y,z,r)\in L^{2,K}_{\mathbb F}(0,\infty;\mathbb R^n) \times
\mathcal L^{2,K}_{\mathbb F}(0,\infty)$. Moreover, $(u_1,u_2)$
defined by \eqref{Sec4.2_NashPoint} is a Nash equilibrium point for
Problem (NZSSDG).
\item If we further assume that, for any $(\omega, t, e)\in \Omega
\times [0,\infty)\times\mathcal E$, any $i=1,2$, $j=1,2,\dots,d$,
$k=1,2,\dots,l$
\begin{equation}\label{Sec4.2_Assumption_Last}
D_i(t)R_i^{-1}(t)D_i^\top(t) \geq \mu I,\quad L_i(t)\geq \mu I,
\quad M_{ij}(t)\geq \mu I,\quad S_{ik}(t,e)\geq \mu I,
\end{equation}
then $(u_1,u_2)$ defined by \eqref{Sec4.2_NashPoint} is the unique
Nash equilibrium point for Problem (NZSSDG).
\end{enumerate}
\end{theorem}

\begin{proof}
(i) Under Assumptions (A3.1)-(A3.3), by Theorem \ref{Sec3_THM}, the
forward-backward SDE \eqref{Sec4.2_FBSDE_Hama} admits a unique
solution in $\mathcal L^{2,K}_{\mathbb F}(0,\infty)$. With the help
of Assumptions (A2.1)-(A2.3) and Lemma \ref{Sec4.2_Lemma_Hama}, the
forward-backward SDE \eqref{Sec4.2_FBSDE_NZSSDG} admits also a
unique solution in the space $L^{2,K}_{\mathbb F}(0,\infty;\mathbb
R^n)\times \mathcal L^{2,K}_{\mathbb F}(0,\infty)$. Moreover, by
Lemma \ref{Sec4.2_Lemma}, $(u_1,u_2)$ defined by
\eqref{Sec4.2_NashPoint} provides a Nash equilibrium point for
Problem (NZSSDG).

(ii) Let $(\bar u_1, \bar u_2)$ be another Nash equilibrium point
for Problem (NZSSDG). From the viewpoint of Lemma
\ref{Sec4.2_Lemma}, for any $i=1,2$, fix $\bar u_{3-i}$, then $\bar
u_i$ is an optimal control of the LQ SOC problem
\eqref{Sec4.2_Cost_i}-\eqref{Sec4.2_Sys_i}. Thanks to
\eqref{Sec4.2_Assumption_Last} and Theorem \ref{Sec4.1_THM}, $\bar
u_i$ must have the form:
\[
\bar u_i(t) = R_i^{-1}(t)D_i^\top(t) x_i(t),\quad t\in [0,\infty),
\]
where $(x_i,y,z)$ satisfies \eqref{Sec4.2_FBSDE_SOC_i}. Combining
the two cases: $i=1$ and $i=2$, we get the conclusion:
\[
\left(\begin{array}{ccc} \bar u_1(t) \\ \bar u_2(t)
\end{array}\right) = \left(\begin{array}{ccc} u_1(t) \\ u_2(t)
\end{array}\right) =
\left(\begin{array}{ccc} R_1^{-1}(t)D_1^\top(t) x_1(t)\\
R_2^{-1}(t)D_2^\top(t) x_2(t)
\end{array}\right),\quad t\in [0,\infty),
\]
where $(x_1,x_2,y,z,r)$ is the unique solution of
\eqref{Sec4.2_FBSDE_NZSSDG}. We complete the proof.
\end{proof}

\begin{example}
In this example, we would like to focus ourselves on a special case
to illustrate the results obtained in this section. Let the
dimension of the state process $y$, the dimension of Brownian motion
and the number of Poisson random measures are $1$. Under this `$1$
dimension' setting, all the involving matrix-valued processes are
real-valued indeed. Moreover, let $A$, $B$, $C$, $D$, $Q$, $L$, $M$,
$S$, $R$ in Problem (SOC) and $A$, $B$, $C$, $D_i$, $Q_i$, $L_i$,
$M_i$, $S_i$, $R_i$ ($i=1,2$) in Problem (NZSSDG) are independent of
the time variable $t$ (and then they are independent of $\omega$
also due to the $\mathbb F$-adaptedness of processes). Furthermore,
let the nonhomogeneous term $\alpha$ in \eqref{Sec4.1_Sys} and
\eqref{Sec4.2_Sys} vanish.

(1). Let $D\neq 0$, $Q\geq 0$, $L>0$, $M>0$, $R>0$, and there exists
a constant $\kappa>0$ such that $S(e)\geq \kappa$ for all
$e\in\mathcal E$. Define
\[
\mu = \min\left\{ \frac{D^2}{R}, L, M, \kappa \right\}.
\]
For any $K\in (0,2\mu)$, Assumptions (A1.1)-(A1.2) are satisfied. By
Theorem \ref{Sec4.1_THM}, the forward-backward SDE
\eqref{Sec4.1_FBSDE_SOC} admits a unique solution, and Problem (SOC)
has a unique optimal control which is given by
\eqref{Sec4.1_OptimalControl}.

(ii). For $i=1,2$, let $D_i\neq 0$, $Q_i\geq 0$, $L_i>0$, $M_i>0$,
$R_i>0$, and there exists a constant $\kappa>0$ such that
$S(e)\geq\kappa$ for all $e\in\mathcal E$. Moreover, we assume
\begin{equation}
A< -\frac 1 2 B^2 -\frac 1 2 \int_{\mathcal E} |C(e)|^2 \pi(de).
\end{equation}
Define
\[
\begin{aligned}
& \mu = \min \bigg\{ \frac{D_1^2}{R_1}, \frac{D_2^2}{R_2}, L_1, L_2,
M_1, M_2, \kappa, \frac{D_1^2}{R_1}L_1 +\frac{D_2^2}{R_2}L_2,
\frac{D_1^2}{R_1}M_1 +\frac{D_2^2}{R_2}M_2, \bigg(\frac{D_1^2}{R_1}
+\frac{D_2^2}{R_2}\bigg)\kappa \bigg\},\\
& \rho = \min\bigg\{-\frac 1 2 B^2 -\frac 1 2 \int_{\mathcal E}
|C(e)|^2 \pi(de)-A, \mu, 1 \bigg\}.
\end{aligned}
\]
For any $K\in (0,2\rho)$, it is easy to check Assumptions
(A2.1)-(A2.3), (A3.1)-(A3.3), and \eqref{Sec4.2_Assumption_Last} are
satisfied. By Theorem \ref{Sec4.2_THM}, the forward-backward SDE
\eqref{Sec4.2_FBSDE_NZSSDG} admits a unique solution, and Problem
(NZSSDG) has a unique Nash equilibrium point which is given by
\eqref{Sec4.2_NashPoint}.
\end{example}

\section{Conclusion}

%Motivated by some potential applications in financial markets,
%especially in the contingent claims pricing and the optimal
%investments and consumptions problems,
In this paper, we investigate a kind of forward-backward stochastic
differential equations (SDEs) driven by both Brownian motions and
Poisson processes on an infinite horizon. In our setting, besides
the coupling of mappings $b$, $\sigma$, $\gamma$ and $g$, the two
initial values are also coupled. We employ a new technique to treat
the coupling between the initial values. For this kind of
forward-backward SDEs, we establish an existence and uniqueness
theorem by virtue of the method of continuation under some
monotonicity conditions. Some important properties including
stability and comparison of solutions are also addressed. These
results generalize that of Peng and Shi \cite{PengShi2000}.

The theoretical results are applied to solve an infinite horizon
linear-quadratic (LQ) backward stochastic optimal control problem
and an LQ nonzero-sum backward stochastic differential game. Under
suitable conditions, we get the solvability of the Hamiltonian
systems related to the LQ control problem and LQ game problem, which
are linear forward-backward SDEs of the type studied previously.
Then the unique optimal control and the unique Nash equilibrium
point are obtained in closed forms, respectively.


\begin{thebibliography}{99}

\bibitem{Antonelli1993}
F. Antonelli, Backward-forward stochastic differential equations,
Ann. Appl. Probab. 3 (1993), no. 3, 777-793.

%\bibitem{Bismut1978}
%J.M. Bismut, An introductory approach to duality in optimal
%stochastic control, SIAM Rev. 20 (1978), no. 1, 62-78.

\bibitem{ContTankov2004}
R. Cont and P. Tankov, Financial modelling with jump processes,
Chapman \& Hall/CRC Financial Mathematics Series, 2004.

\bibitem{CvitanicMa1996}
J. Cvitani\'{c} and J. Ma, Hedging options for a large investor and
forward-backward SDE's, Ann. Appl. Probab. 6 (1996), no. 2, 370-398.

\bibitem{Hamadene1998}
S. Hamad\`{e}ne, Backward-forward SDE's and stochastic differential
games, Stochastic Process. Appl. 77 (1998), no. 1, 1-15.

\bibitem{HuPeng1995}
Y. Hu and S. Peng, Solution of forward-backward stochastic
differential equations, Probab. Theory Related Fields 103 (1995),
no. 2, 273-283.

%\bibitem{HuangLiYong2012}
%J. Huang, X. Li and J. Yong, A linear-quadratic optimal control
%problem for mean-field stochastic differential equations in infinite
%horizon, Math. Control Relat. Fields 5 (2015), no. 1, 97-139.

\bibitem{MaProtterYong1994}
J. Ma, P. Protter and J. Yong, Solving forward-backward stochastic
differential equations explicitly-a four step scheme. Probab. Theory
Related Fields 98 (1994), no. 3, 339-359.

\bibitem{MWZZ2011}
J. Ma, Z. Wu, D. Zhang and J. Zhang, On wellposedness of
forward-backward SDEs-a unified approach, Ann. Appl. Probab.  25
(2015),  no. 4, 2168-2214.

\bibitem{MaYong1999}
J. Ma and J. Yong, Forward-backward stochastic differential
equations and their applications. Lecture Notes in Mathematics,
1702. Springer-Verlag, Berlin, 1999.

\bibitem{OS2007}
B. {\O}ksendal and A. Sulem, Applied stochastic control of jump
diffusions, Second edition. Universitext. Springer, Berlin, 2007.

%\bibitem{PardouxPeng1990}
%\'{E}. Pardoux and S. Peng, Adapted solution of a backward
%stochastic differential equation, Systems Control Lett. 14 (1990),
%no. 1, 55-61.

%\bibitem{PardouxPeng1992}
%\'{E}. Pardoux and S. Peng, Backward stochastic differential
%equations and quasilinear parabolic partial differential equations.
%Stochastic partial differential equations and their applications
%(Charlotte, NC, 1991), 200-217, Lecture Notes in Control and Inform.
%Sci., 176, Springer, Berlin, 1992.

\bibitem{PardouxTang1999}
\'{E}. Pardoux and S. Tang, Forward-backward stochastic differential
equations and quasilinear parabolic PDEs. Probab. Theory Related
Fields 114 (1999), no. 2, 123-150.

%\bibitem{Peng1991}
%S. Peng, Probabilistic interpretation for systems of quasilinear
%parabolic partial differential equations, Stochastics Stochastics
%Rep. 37 (1991), no. 1-2, 61-74.

%\bibitem{Peng2000}
%S. Peng, Problem of eigenvalues of stochastic Hamiltonian systems
%with boundary conditions. Stochastic Process. Appl. 88 (2000), no.
%2, 259-290.

\bibitem{PengShi2000}
S. Peng and Y. Shi, Infinite horizon forward-backward stochastic
differential equations, Stochastic Process. Appl. 85 (2000), no. 1,
75-92.

\bibitem{PengWu1999}
S. Peng and Z. Wu, Fully coupled forward-backward stochastic
differential equations and applications to optimal control, SIAM J.
Control Optim. 37 (1999), no. 3, 825-843.

\bibitem{ShenMengShi2014}
Y. Shen, Q. Meng and P. Shi, Maximum principle for mean-field
jump-diffusion stochastic delay differential equations and its
application to finance, Automatica J. IFAC 50 (2014), no. 6,
1565-1579.

\bibitem{Wu2003}
Z. Wu, Fully coupled FBSDE with Brownian motion and Poisson process
in stopping time duration, J. Aust. Math. Soc. 74 (2003), no. 2,
249-266.

\bibitem{WuYu2014}
Z. Wu and Z. Yu, Probabilistic interpretation for a system of
quasilinear parabolic partial differential equation combined with
algebra equations, Stochastic Process. Appl. 124 (2014), no. 12,
3921-3947.

\bibitem{Yin2008}
J. Yin, On solutions of a class of infinite horizon FBSDEs, Statist.
Probab. Lett. 78 (2008), no. 15, 2412-2419.

\bibitem{Yin2011}
J. Yin, Forward-backward SDEs with random terminal time and
applications to pricing special European-type options for a large
investor, Bull. Sci. Math. 135 (2011), no. 8, 883-895.


\bibitem{Yong1997}
J. Yong, Finding adapted solutions of forward-backward stochastic
differential equations: method of continuation. Probab. Theory
Related Fields 107 (1997), no. 4, 537-572.

\bibitem{Yong2010}
J. Yong, Forward-backward stochastic differential equations with
mixed initial-terminal conditions. Trans. Amer. Math. Soc. 362
(2010), no. 2, 1047-1096.

\bibitem{YongZhou1999}
J. Yong and X. Zhou, Stochastic controls. Hamiltonian systems and
HJB equations. Applications of Mathematics (New York), 43.
Springer-Verlag, New York, 1999.

\bibitem{Yu2012}
Z. Yu, Linear-quadratic optimal control and nonzero-sum differential
game of forward-backward stochastic system, Asian J. Control 14
(2012), no. 1, 173-185.

\bibitem{Zeidler1990}
E. Zeidler, Nonlinear functional analysis and its applications,
II/B, Nonlinear monotone operators, Translated from the German by
the author and Leo F. Boron, Springer-Verlag, New York, 1990.


\end{thebibliography}
\end{document}